\documentclass{amsart}
\usepackage{amsfonts}
\usepackage{latexsym}
\usepackage{amssymb}
\usepackage{amsmath}
\usepackage{enumerate}
\usepackage{latexsym,color,amsmath,amsthm,amssymb,amscd,amsfonts}
\usepackage{graphicx}
\usepackage[normalem]{ulem}

\usepackage[utf8]{inputenc}
\usepackage[T1]{fontenc}

\DeclareMathOperator*{\interior}{int}

\newcommand{\F}{\mathbb F}

\newcommand{\R}{\mathbb R}
\newcommand{\N}{\mathbb N}

\renewcommand{\H}{\mathbb H}

\newcommand{\elip}{\mathcal E}

\newtheorem{thm}{Theorem}[section]

\newtheorem{lemma}[thm]{Lemma}

\newtheorem{example}[thm]{Example}
\newtheorem{propo}[thm]{Proposition}
\theoremstyle{remark}
\newtheorem*{rmk}{Remark}

\newtheorem{problem}{Problem}
\begin{document}

%%%%%%%%%%%%%%%%%%%%%%%%%%%%%%%%%%%%%%%%%%%%%5

\title{Best approximation of functions by log-polynomials}

\author[R. Alonso-Guti\'errez]{David Alonso-Guti\'errez}
\email{alonsod@unizar.es}
\address{Departamento de Matemáticas. Universidad de Zaragoza, Spain}

\author[B. Gonz\'alez Merino]{Bernardo Gonz\'alez Merino}
\email{bgmerino@um.es}
\address{Departamento de Matem\'atica Aplicada, 
%Facultad de Inform\'atica, 
Universidad de Murcia, 
%30100-Murcia, 
Spain}

\author[R. Villa]{Rafael Villa}
\email{villa@us.es}
\address{Departamento de Análisis Matemático. Universidad de Sevilla, Spain}


\thanks{This research is a result of the activity developed within the framework of the Programme in Support of Excellence Groups of the Regi\'on de Murcia, Spain, by Fundaci\'on S\'eneca, Science and Technology Agency of the Regi\'on de Murcia. 
The first  author is partially supported by MICINN Project PID-105979-GB-I00 and DGA Project E48\_20R. 
The second author is partially supported by Fundaci\'on S\'eneca project 19901/GERM/15, Spain. 
The second and third authors are partially supported by MICINN Project PGC2018-094215-B-I00 Spain.}

\date{\today}

\begin{abstract}
Lasserre \cite{lasserre} proved that for every compact set $K\subset\R^n$ and every even number $d$ there exists
a unique homogeneous polynomial $g_0$ of degree $d$ with $K\subset G_1(g_0)=\{x\in\R^n:g_0(x)\leq 1\}$
minimizing $|G_1(g)|$ among all such polynomials $g$ fulfilling the condition $K\subset G_1(g)$. This result extends the notion of the Löwner ellipsoid, not only from convex bodies  to arbitrary compact sets (which was immediate if $d=2$ by taking convex hulls), but also from ellipsoids to level sets of homogeneous polynomial of an arbitrary even degree.
%for convex bodies.
%DA:He cambiado la frase en la que ponía esto: (or convex hulls of compact sets)

In this paper we extend this result for the class
of non-negative log-concave functions in two different ways. One of them is the straightforward extension of the known results, and the other one is a suitable extension with uniqueness 
of the solution in the corresponding problem
and a characterization in terms of some 'contact points'.

%In this work we prove that for any $f:\R^n\rightarrow[0,+\infty)$
%integrable of compact support with $f(0)=||f||_\infty=1$
%there exist $t_0\geq 1$ and a homogeneous polynomial $g_0$ of even degree $d$
%with $f(x)\leq t_0e^{-g_0(x)}$ such that
%\[
%\int_{\R^n}t_0e^{-g_0(x)^\frac{1}{d}}dx
%\]
%is the minimum amongst the integral of all functions $te^{-g(x)^\frac{1}{d}}$
%greater than or equal to $f(x)$ and where $t\geq 1$ and $g$ is a homogeneous polynomial of degree $d$.

%This result can be seen as the functional version of the work of Lasserre \cite{lasserre}, it extends the results in
%\cite{agjv} to a more general set of functions, and in particular establishes the functional L\"owner ellipsoid theorem.
\end{abstract}

\maketitle

\section{Introduction}

In
\cite{agjv}
the authors proved that for any non-negative integrable log-concave function $f:\R^n\to[0,+\infty)$ with $\Vert f\Vert_\infty=f(0)=1$, there exists a
unique pair $(t_0,\elip_0)$, with $t_0\in(0,1]$ and $\elip_0\subset\R^n$ an ellipsoid such that
\begin{equation}\label{ineq:func0}
t_0\chi_{\elip_0}\le f
\end{equation}
maximizing
$$
\int_{\R^n} t\chi_{\elip}(x)\,dx=t|\elip|
$$
among all the pairs $(t,\elip)$ verifying (\ref{ineq:func0}), where  for any measurable set $K\subseteq\R^n$, $\chi_K$ denotes its characteristic function and $|K|$ denotes its $n$-dimensional Lebesgue measure.

% of $K\subset\R^n$.
This is a functional version of John's celebrated theorem \cite{J},
%\textcolor{red}{DA: ?`Esta es la referencia que queremos poner?Antes pon\'ia una de Henk}
 which provides the existence of a unique maximal volume ellipsoid contained in any convex body
$K\subset\R^n$. This maximal volume ellipsoid is called the John ellipsoid of $K$. Whenever one takes, for any convex body $K\subseteq\R^n$ containing the origin, $f=\chi_K$ in the aforementioned maximization the solution is $t_0=1$ and $\elip_0$ is the John ellipsoid of $K$.
For other functional version of the problem, see \cite{IN}.

% in (\ref{ineq:func0}) (in that case, $t_0=1$).

Recall that if $f(x)=e^{-u(x)}$ with $u:\R^n\to[0,+\infty]$ a convex function, its polar function $f^\circ$ is defined as $f^\circ(x)=e^{-u^*(x)}$, where $u^*:\R^n\to\R$ is the Legendre transform of $u$ given by $u^*(x)=\sup_{y\in\R^n}(\langle x,y\rangle-u(y))$.  If $K\subseteq\R^n$ is a convex body containing the origin in its interior, then $(\chi_K)^\circ=e^{-\Vert\cdot\Vert_{K^\circ}}$. Here, for any convex body $K$ containing the origin, $\|\cdot\|_{K}$ denotes the Minkowski gauge
$$
\Vert x\Vert_K=\inf\{\lambda>0\,:\,x\in\lambda K\},
$$
and $K^\circ$ denotes the polar body of $K$, defined by $$K^\circ=\{x\in\R^n\,:\,\langle x,y\rangle\leq1,\,\forall y\in K\}.$$ Besides, we will denote by
$|x|=\|x\|_{B^n_2}$ the Euclidean norm, for every $x\in\mathbb R^n$, where $B^n_2$ denotes the Euclidean unit ball and the epigraph of a convex function $u:\R^n\to[0,+\infty)$ by
$$
\textrm{epi}(u):=\{(x,t)\in\R^n\times[0,+\infty)\,:\,u(x)\leq t\}.
$$

Using this functional notion of polarity \cite{km}, taking into account that $f^{\circ\circ}=f$ for any log-concave function $f$, and the equality $|\elip||\elip^\circ|=|B_2^n|^2$  for any origin centered ellipsoid $\elip$, whenever $f$ is even the previous result can be stated as a minimizing volume problem,
i.e., for any even integrable log-concave function $f:\R^n\to[0,+\infty)$ with $\Vert f\Vert_\infty=f(0)=1$ there exists a unique pair $(t_1,\elip_1)$ with $t_1\ge1$  and $\elip_1\subset\R^n$ a origin centered ellipsoid such that
\begin{equation}\label{ineq:func1}
f\le t_1\exp\{-\|\cdot\|_{\elip_1}\}
\end{equation}
minimizing
$$
\int t_1\exp\{-\|x\|_{\elip_1}\}\,dx=t_1n!|\elip_1|
$$
among all the pairs $(t,\elip)$ verifying \eqref{ineq:func1}.
%Here, for any convex body $K$ containing the origin, $\|\cdot\|_{K}$ denotes the Minkowski gauge
%$$
%\Vert x\Vert_K=\inf\{\lambda>0\,:\,x\in\lambda K\}.
%$$
%Moreover, let
%$|x|=\|x\|_{B^n_2}$ be the Euclidean norm, for every $x\in\mathbb R^n$, where %$B^n_2$ denotes the Euclidean unit ball.

In \cite{LSW} the authors provided a definition of a functional L\"owner ellipsoid also whenever $f$ is not necessarily even. They considered the corresponding integral minimization problem related to such functional ellipsoid.
%The solution also appears as the \textbf{unique} solution of a minimization problem. 
In 
%such 
that case the solution does not necessarily coincide with the polar of the functional John ellipsoid  of the polar function.
This result generalizes the dual version of John's Theorem, which states that for any convex body $K\subseteq\R^n$ there exists a unique ellipsoid, known as  the L\"owner ellipsoid of $K$, of minimal volume containing $K$. Whenever one takes, for any convex body $K$ containing the origin, $f=\exp\{-\|\cdot\|_{K}\}$, the solution of the minimization problem appearing in \cite{LSW} also recovers the L\"owner's ellipsoid of $K$. %\eqref{ineq:func1} (in that case, $t_1=1$,
%and possibly we have applied a suitable translation of $K$).
Interpreting and proving functional versions for log-concave functions of well-known geometric results has become increasingly popular in the last years, see for instance \cite{ABG}, \cite{agjv}, \cite{AGJV2}, \cite{AFS}, \cite{AKM}, \cite{AKSW}, \cite{BL}, \cite{CW}, \cite{C}, \cite{CF}, \cite{FM}, \cite{FZ}, \cite{km}, \cite{Li}, \cite{Ro}.

%All these  results characterize the optimal solution with some touching conditions involving  contact points.
John and L\"owner ellipsoids of convex bodies have been widely investigated in the literature (see, for example, \cite{GLMP}, \cite{GS}, \cite{GPT}, \cite{He}). Furthermore,
 the John or the L\"owner ellipsoid of a convex body $K\subseteq\R^n$ is characterized by the existence of some contact points between the boundary of $K$ and the Euclidean sphere, $S^{n-1}$ (see \cite{B}, \cite{BR}).

%%% para duke
Other connections between convex bodies and  ellipsoids can be found in the literature. For instance, the Legendre and Binet ellipsoids are well-known concepts from classical mechanics. For some references, see
\cite{Le},
\cite{LiM},
\cite{MPa1}, 
\cite{MPa2}, and  \cite{LYZ} for recent developments.
%%% para duke

On the other hand, Laserre \cite{lasserre} generalized the definition of the L\"owner ellipsoid   for any compact (non-necessarily convex) set by means of
replacing  the bilinear form given by an ellipsoid  by a homogeneous polynomial of even degree $d\ge2$.

More precisely, if we denote by $\H_d(\R^n)$ the vector space of homogeneous polynomials of degree $d$ in $\mathbb R^n$, of dimension $h_d(n)={n+d-1\choose d}$, it was proved that, given any compact set $K\subset\R^n$ with non-empty interior and an even integer $d\in\N$,
there exists a unique homogeneous polynomial $g_0\in \H_d(\R^n)$ of degree $d$, the \emph{d-Lasserre-L\"owner polynomial},
such that
\begin{equation}\label{ineq:set}
K
\subseteq
G_1(g_0)=\{x\in\R^n: g_0(x)\le 1\}
\end{equation}
with minimum volume $|G_1(g_0)|$ among all $d$-homogeneous polynomial verifying \eqref{ineq:set}.

Let $\F_d(\R^n)$ be the set in $\H_d(\R^n)$ of all $d$-homogeneous polynomials in $\R^n$ such that $|G_1(g)|<+\infty$.
Note that $|G_1(g)|<+\infty$ implies $g\ge0$. In particular, the previous minimization problem cannot be stated for odd $d$.

Moreover, the solution is also characterized in terms of some common contact points in the boundaries of $K$ and $G_1(g_0)$ (cf.~\cite{lasserre}). More precisely,
 $|G_1(g_0)|$ is minimum among all $g\in\H_d(\R^n)$ verifying \eqref{ineq:set} if and only if
there exist $y_1,\dots,y_s\in K$, $\lambda_1,\dots,\lambda_s>0$, with
$s\leq h_d(n)$, such that $g_0(y_i)=1$ for $i=1,\dots,s$, and
\begin{equation}\label{eq:set_touching}
\int_{\mathbb R^n}x^\alpha e^{-g_0(x)}dx=\sum_{i=1}^s\lambda_iy_i^\alpha
\end{equation}
for every $\alpha\in\N^n$ such that $|\alpha|=\sum_{i=1}^n\alpha_i=d$, where $x^\alpha= x_1^{\alpha_1} \cdots x_n^{\alpha_n}$.
Note that the identity above implies a trace identity (see Lemma \ref{identity_r})
%%\textcolor{red}{Esto no lo veo claro. Igual hace falta alguna explicaci\'on m\'as.}
$$
\frac{n}{d}\int_{\mathbb R^n}e^{-g_0(x)}dx=\int_{\mathbb R^n}g_0(x)e^{-g_0(x)}dx=\sum_{i=1}^s\lambda_i.
$$

In this paper we will extend the result of \cite{lasserre} to the functional setting.
Let us pose the following problem:
\begin{problem}
\label{problem1}
Given $f:\R^n\to[0,+\infty)$   with $\|f\|_\infty=f(0)=1$, and $d\in\N$ even,
minimize
$$
\int_{\mathbb R^n} te^{-g(x)^{\frac{1}{d}}}\,dx=tn!|G_1(g)|
$$
among all $g\in\H_d(\R^n)$  and  $t\geq 1$   such that
\begin{equation}
\label{ineq:func_poly1}
f(x)\leq te^{-g(x)^{\frac{1}{d}}}.
\end{equation}
\end{problem}

Note that the functional to be optimized verifies a strong global convexity property
on the space of pairs $(r,g)$ verifying \eqref{ineq:func_poly1},
once the natural reparametrization $t=e^r$, together with an appropriate change in the integral to consider, is taken
(see Lemma \ref{lem:IntegralOfg} and Lemma \ref{lemma:global_convex}).
Despite this global property, the set of pairs $(r,g)$ verifying \eqref{ineq:func_poly1}, with $t=e^r$, does not
verify a suitable convexity or compactness property, so the existence and uniqueness of a minimizing pair is not straightforwardly obtained.
Considering, for instance, $f(x)=\chi_{B_2^n}$ and taking $r_0,r_1>0$ and polynomials of the form $g_i(x)=r_i^d|x|^{d}$, $i=0,1$, we have
\[
\chi_{B^n_2} \leq \exp\{r_i-g_{i}^{1/d}\}
\]
 for $i=0,1$. However, 
\[
\chi_{B^n_2} \nleq \exp\{r_\theta-g_{\theta}^{1/d}\}
\]
for any $\theta\in(0,1)$, where $r_\theta=(1-\theta)r_0+\theta r_1$ and $g_\theta=(1-\theta)g_0+\theta g_1$.

In view of Lasserre's result, one might think in integrable functions $f:\R^n\rightarrow[0,+\infty)$
as the typical extension of compact sets to spaces of functions. Unfortunately, \textit{Problem \ref{problem1}}
does not make sense in such a general case (see Example \ref{example:no_integ}). Motivated by this fact, we solve the problem in the setting of
 log-concave integrable functions. Let   $\mathcal F(\mathbb R^n)$ be the set of all log-concave integrable functions on $\mathbb R^n$.

\begin{thm}\label{thm:prob1}
Let $f\in\mathcal F(\mathbb R^n)$ with $\Vert f\Vert_\infty=f(0)=1$ and $d\in\mathbb N$ even.
Then there exists $(t_1,g_1)\in[1,+\infty)\times\F_d(\R^n)$ a solution of  \textit{Problem \ref{problem1}}.
\end{thm}

Notice that the problem considered in Theorem \ref{thm:prob1} was solved with uniqueness for $d=2$  in \cite{agjv}, in the even case, and in \cite{LSW}, in the general case,
since for $g\in\H_2(\mathbb R^n)$  the set $G_1(g)$ is an ellipsoid provided that $|G_1(g)|$ is finite (see Lemma \ref{F_d} (\ref{prop:ellip})).

However, even this case is not solved with uniqueness in the proof of Theorem \ref{thm:prob1} with this point of view, since polarity does not work clearly between polynomials.
More precisely, if we try to construct a proof by taking duals in the proof in \cite{agjv}, we would need to take the polar of the ellipsoid $G_1(g)$, but the expression of the polynomial defining the polar ellipsoid in terms of $g$ is not clear.

In the general case,
the uniqueness is not straightforwardly obtained (although we do not know any example for which the minimization point is not unique). It seems to us that the proof would require some more convexity properties than the ones we have obtained.

%In the general case,
% uniqueness would require some more convexity properties than the ones we have (we do not know any example for which the minimization point is not unique).

%Trying to take uniqueness into consideration, 
The following similar problem is also posed. Unlike the case of Problem \ref{problem1} above, we are able to show
existence and uniqueness of the solution.

 \begin{problem}
\label{problem2}
Given $f:\R^n\to[0,+\infty)$   with $\|f\|_\infty=f(0)=1$, and $d\in\N$ even,
minimize
$$
\int_{\mathbb R^n} te^{-g(x)}\,dx=t\Gamma(\tfrac{n}{d}+1)|G_1(g)|
$$
among all $g\in\H_d(\R^n)$  and  $t\geq 1$   such that
\begin{equation}
\label{ineq:func_poly2}
f(x)\leq te^{-g(x)}.
\end{equation}
\end{problem}

Again,the existence of a global minimum is not guaranteed using the convexity of the functional to be optimized, since
the argument would need the feasible set of solutions to be a convex compact set. Lemma \ref{lem:log-convex-varphi}
proves that the set is convex, but compactness can not be assured.

This problem is solved with uniqueness, when imposing an extra condition, in the following result.

\begin{thm}\label{thm:prob2}
Let $f:\R^n\to[0,+\infty)$  be a log-concave function  with $\|f\|_\infty=f(0)=1$ and $d\in\mathbb N$ even,
such that
% there is a pair $(t,g)\in[1,+\infty)\times\F_d(\R^n)$ verifying \eqref{ineq:func_poly2} with $G_1(g)$ bounded.
$$
\widehat H_1(f)
=\bigcup_{\lambda\in(0,1)}\log(1/\lambda)^{-1/d}\{x\in\R^n: f(x)\ge \lambda\}
$$
is bounded.
Then
there exists $(t_2,g_2)\in[1,+\infty)\times\F_d(\R^n)$ a unique solution of \textit{Problem \ref{problem2}}.
\end{thm}

For an interior minimization point  $(t_2,g_2)\in(1,+\infty)\times\textrm{int}\F_d(\R^n)$, being the
unique solution of \textit{Problem \ref{problem2}} can be characterized by some \textit{touching conditions}, via the Karush-Kuhn-Tucker conditions (see \cite{APE}, \cite{hiriart}). For these conditions to hold, no hypothesis on the log-concavity of $f$ is needed.

\begin{thm}\label{thm:touching2}
Let $f:\R^n\to[0,+\infty)$  be a bounded function  with  $\|f\|_\infty=f(0)=1$.
%, with $f(x)\leq \exp(r-g(x))$, for some $(r,g)\in\R_+\times\R^N$ with $G_1(g)$ bounded.
Moreover, let
$(t_2,g_2)\in(1,+\infty)\times\mathrm{int}(\F_d(\R^n))$
 be such that
$f(x)\leq t_2\exp(-g_2(x))$ for every $x\in\mathbb R^n$.
Then the following are equivalent:
\begin{enumerate}[(i)]
\item $(t_2,g_2)$ is the only solution of \textit{Problem \ref{problem2}}.
\item There exist $x_1,\dots,x_m\in\mathbb R^n$, $m\leq {n+d-1\choose d}+1$, with
$f(x_i)=t_2\exp(-g_2(x_i))$, and $\lambda_i>0$, $1\leq i\leq m$, such that
\begin{align*}
t_2\int_{\R^n}\exp(-g_2(x))dx&=\sum_{i=1}^m\lambda_i \text{  and}
\\
t_2\int_{\R^n}x^\alpha\exp(-g_2(x))dx&=\sum_{i=1}^m\lambda_ix_i^\alpha\
\text{  for all  }
\alpha\in\N^n_d.
%|\alpha|=d.
\end{align*}
\end{enumerate}
\end{thm}

The paper is organized as follows. In Section \ref{sec:TechnicalResults} we provide all the definitions and properties related to homogeneous polynomials which are needed for the study of both problems. Section \ref{section:prob1} is devoted to give the existence of a minimization point in \textit{Problem \ref{problem1}}. In Section \ref{section:prob2}
we study \textit{Problem \ref{problem2}}, giving similar results as the ones given in Section \ref{section:prob1}, and new facts that allow to prove the existence and uniqueness of the minimization problem, under the  additional assumption given in Theorem \ref{thm:prob2}. Further, we give the characterization of the minimization point in terms of the contact points.
Finally in Section \ref{sec:Application_OVR_d_oir_d} we introduce the $d$-outer volume and integral ratio of a 
convex body, and show an application of the $d$-L\"owner-Lasserre polynomial to approximation of convex bodies.

\section{Homogeneous polynomials}\label{sec:TechnicalResults}

Let $\H_d(\R^n)$
be the vector space of homogeneous polynomials of degree $d$ in $\mathbb R^n$, with dimension $h_d(n)={n+d-1\choose d}$. Any $g\in \H_d(\R^n)$ can be uniquely written as
$$
g(x)=\sum_{\alpha\in\N_d^n} g_\alpha x^\alpha
$$
where $\N_d^n=\{\alpha=(\alpha_1,\dots,\alpha_n)\in\N^n: |\alpha|=\alpha_1+\cdots+\alpha_n=d\}$ and for $x=(x_1,\dots,x_n)\in\R^n$ and $\alpha\in\N_d^n$, $x^\alpha$ denotes the monomial $\displaystyle{x^\alpha=\prod_{i=1}^n x_i^{\alpha_i}}$.%x_1^{\alpha_1}\cdot\cdots\cdot x_n^{\alpha_n}$.

For any $g\in\H_d(\R^n)$, let us denote, for any $t>0$, $G_t(g)=\{x\in\R^n: g(x)\le t\}$. Notice that by the homogeneity of $g$, $G_t(g)=t^\frac{1}{d}G_1(g)$, and that if $|G_1(g)|<+\infty$ then necessarily $g$ must be non-negative on $\R^n$ and therefore $d$ must be even. Moreover, if $d=2$, $G_1(g)$ is an ellipsoid. However, for $d>2$, $G_1(g)$ can be non-convex, and even unbounded, as the example $g(x,y)=x^2 y^2(x^2+y^2)$ shows (see \cite{lasserre} and Lemma \ref{F_d} below).

Let $\F_d(\R^n)$ be the set in $\H_d(\R^n)$ of all $d$-homogeneous polynomials in $\R^n$ such that $|G_1(g)|<+\infty$.

 We first show the identities involving the integrals in \textit{Problem \ref{problem1}} and \textit{Problem \ref{problem2}} and the volume $|G_1(g)|$. They are particular cases of the following technical result (a particular case is given in~\cite[Thm.~2.2]{lasserre}).

\begin{lemma} \label{identity_r} Let $n\ge1$, $k\ge0$, $d\ge2$ even, $r\in\R$ and $m>0$ be such that
$\frac{n+k}{d}+r>0$. For $\alpha\in\N_k^n$, let $g\in\mathbb H_d(\R^n)$ be such that $x^\alpha$ is integrable in $G_1(g)$. Then
$$
\int_{\R^n} x^\alpha g(x)^r\exp(-g(x)^{1/m})\,dx
=
\tfrac{n+k}{d}m
\Gamma\left(m\left(\tfrac{n+k}{d}+r\right)\right)
\int_{G_1(g)} x^\alpha \,dx.
$$
In particular,
$$
\int_{\R^n}  g(x)\exp(-g(x))\,dx
=
\tfrac{n}{d}
\Gamma\left(\tfrac{n}{d}+1\right)
|G_1(g)|.
$$
\end{lemma}

\begin{proof}
Let us define, for any $y>0$, $w_\alpha(y)=\displaystyle\int_{\{x:g(x)\le y\}} x^\alpha dx$. By the homogeneity of $g$ we have that
$
w_\alpha(y)
%=\int_{\{x:g(y^{-1/d} x)\le 1\}} x^\alpha dx
=y^{\frac{n+k}{d}}w_\alpha(1).
$
Therefore
\begin{align*}
\int_{\R^n} x^\alpha g(x)^r&\exp(-g(x)^{1/m})\,dx
=
\int_{\R^n} x^\alpha \int_{g(x)^{1/m}}^{+\infty} (y-mr)y^{mr-1}e^{-y}\,dy\,dx
\\
&=
\int_0^{+\infty} (y-mr)y^{mr-1}e^{-y}
\int_{\{x:g(x)\le y^m\}} x^\alpha dx\,dy
\\
&=
w_\alpha(1) \int_0^{+\infty} (y-mr)y^{m(\frac{n+k}{d}+r)-1} e^{-y}\,dy
\\
&=
\frac{ n+k}{d}m
\Gamma\left(m\left(\tfrac{n+k}{d}+r\right)\right)w_\alpha(1).\qedhere
\end{align*}
\end{proof}

\begin{lemma}\label{lem:IntegralOfg}
Let $g\in\F_d(\R^n)$. For every $t\ge 0$, and $m>0$,
$$
|G_t(g)|
=
\frac{t^{n/d}}{\Gamma\left(\frac{n m}{d}+1\right)}\int_{\R^n}\exp(-g(x)^{1/m})\,dx.
$$
In particular,
$$
|G_1(g)|
=
\frac{1}{\Gamma\left(\frac{n }{d}+1\right)}\int_{\R^n}\exp(-g(x))\,dx
=
\frac{1}{n!}\int_{\R^n}\exp(-g(x)^{1/d})\,dx.
$$
\end{lemma}

\begin{proof}
By the homogeneity of $g$, $G_t(g)=t^\frac{1}{d}G_1(G)$ and then $|G_t(g)|=t^{n/d}|G_1(g)|<+\infty$ for any $t>0$. Besides, by Lemma \ref{identity_r} with $k=0$,  $\alpha=(0\dots,0)$, and $r=0$, we have that for any $m>0$
$$
\int_{\R^n} \exp(-g(x)^{1/m})\,dx=\Gamma\left(\frac{n m}{d}+1\right)|G_1(g)|.
$$
In particular, taking $m=1$ or $m=d$ we obtain
$$
|G_1(g)|
=
\frac{1}{\Gamma\left(\frac{n }{d}+1\right)}\int_{\R^n}\exp(-g(x))\,dx
=
\frac{1}{n!}\int_{\R^n}\exp(-g(x)^{1/d})\,dx.\qedhere
$$
%Indeed, Lemma \ref{identity_r} with $\alpha=0$, $r=0$ and $m=d$ gives
%$$
%\int_{\R^n} \exp(-g(x)^{1/d})\,dx
%=
%n!
%\int_{\{x:g(x)\le1\}}  dx=n! |G_1(g)|
%$$
%and
%$
%|G_t(g)|=w_0(t)=t^{n/d}|G_1(g)|.
%$
\end{proof}

The following result states some topological properties of $\F_d(\R^n)$.

\begin{lemma} \label{F_d} Let $d\in\N$ be an even integer.
\begin{enumerate}
\item $\F_d(\R^n)$ is a convex cone in $\H_d(\R^n)$, which is not closed and has non-empty interior.

\item\label{prop:ellip} For $d=2$, $g\in\F_2(\R^n)$ if and only if $G_1(g)$ is bounded (an ellipsoid). Moreover, $\F_2(\R^n)$ is open.

\item For  $d=4$, $n=2$, $g\in\F_4(\R^2)$ if and only if $G_1(g)$ is bounded. Moreover $\F_4(\R^2)$ is open.

\item For  $d=4$, $n\ge3$, there exists $g\in\F_4(\R^n)$ so that $G_1(g)$ is not bounded. Moreover $\F_4(\R^n)$ is not open.

\item For $d\ge6$, $n\ge2$,  there exists $g\in\F_d(\R^n)$ so that $G_1(g)$ is not bounded. Moreover, $g\in\F_d(\R^n)$ is not open.
\end{enumerate}\end{lemma}

\begin{proof} \begin{enumerate}
\item $\F_d(\R^n)$ is a convex cone, as proved in \cite[Lemma 2.1]{lasserre}.

The polynomial $g_0(x)=\sum_{i=1}^n  x_i^d$ is an interior point in $\F_d(\R^n)$. In fact, if $g(x)=\sum_{\alpha\in\N_d^n} g_\alpha x^\alpha$ is such that $|g_\alpha-g_{0,\alpha}|<\varepsilon$, for some $\varepsilon<\frac{1}{2\left(1+\left({n+d-1\choose d}-n\right)\right)}$ for every $\alpha\in\N_d^n$, then
\begin{align*}
g(x)&\ge (1-\varepsilon)\sum_{i=1}^n  x_i^d-\varepsilon \left({n+d-1\choose d}-n\right)\Vert x\Vert_\infty^d\cr
&\geq\left(1-\varepsilon\left(1+\left({n+d-1\choose d}-n\right)\right)\right)\Vert x\Vert_\infty^d\cr
&\geq\frac{1}{2}\Vert x\Vert_\infty^d,
\end{align*}
where $\Vert x\Vert_\infty=\max\{|x_i|\,:\,1\leq i\leq n\}$. Then, $G_1(g)$ is bounded, so $g\in\F_d(\R^n)$.

The polynomial 
%$x_1^d+t\sum_{i=2}^n  x_i^d$ 
$tg_0$ belongs to $\F_d(\R^n)$ for any $t>0$ but the zero polynomial does not belong to $\F_d(\R^n)$. 
%Thus, $x_1^d$ lies in the boundary of $\F_d(\R^n)$, but it does not belong to $\F_d(\R^n)$. 
Therefore $\F_d(\R^n)$ is not closed.

%\textcolor{red}{Comprobar. Lo he cambiado bastante}
\item Applying  Sylvester's law of inertia \cite{Sylvester} for quadratic forms, any $g\in\H_2(\R^n)$ can be written in the canonical form $g(x)=\sum_{i=1}^n \alpha_i x_i^2$  with an appropriate change of coordinates. Then $g\in\F_2(\R^n)$ if and only if $\alpha_i>0$ for all $i=1,\dots n$ (if and only if $G_1(g)$ is an ellipsoid). That clearly implies that $\F_2(\R^n)$ is open.

\item Similarly, any $g\in\H_4(\R^2)$ can be written, with an appropriate change of coordinates, in the
canonical form $g(x,y)=ax^4+2bx^2y^2+cy^4$ (see \cite[Les.~XV]{Sa}).

Notice that written in this canonical form $G_1(g)$ is bounded if and only if $a,c>0$ and $b>-\sqrt{ac}$. Indeed, if $G_1(g)$ is bounded then necessarily $a>0$ and $c>0$. In such case, if $b\leq-\sqrt{ac}$ then $b^2\geq ac$ and there exists some $\lambda=-\frac{b}{2a}>0$ such that
$$
a+2b\lambda+c\lambda^2\leq0
$$
and then all the points $(x,y)\in\R^2$ with $y=\sqrt{\lambda}x$ belong to $G_1(g)$. Conversely, if $a,c>0$ and $h=b+\sqrt{ac}>0$, writing
$$
g(x,y)=(\sqrt{a}x^2-\sqrt{c}y^2)^2+2hx^2y^2
$$
we have that if $(x,y)\in G_1(g)$, then $|\sqrt{a}x^2-\sqrt{c}y^2|\le1$, and $2hx^2y^2\le1$. The two inequalities imply $|x|,|y|$ are bounded and then $G_1(g)$ is bounded.

Furthermore, $g\in\F_4(\R^2)$ if and only if $a,c>0$ and $b>-\sqrt{ac}$.
Indeed, if $|G_1(g)|<+\infty$, then $a,c>0$, otherwise for every $y_0\in\R$, $g(x,y_0)\leq0$ for every $x$ large enough, or for every $x_0\in\R$ $g(x_0,y)\leq0$ for every $y$ large enough; in any case $\iint e^{-g(x,y)}\,dxdy=+\infty$. If $a,c>0$ and $b\le -\sqrt{ac}$, then $g(x,y)\le(\sqrt{a}x^2-\sqrt{c}y^2)^2=(\sqrt{a}x+\sqrt{c}y)^2(\sqrt{a}x-\sqrt{c}y)^2$. The change of variables $u=\sqrt{a}x+\sqrt{c}y$, $v=\sqrt{a}x-\sqrt{c}y$ and the fact $\iint e^{-u^2v^2}\,dudv=+\infty$ show that $|G_1(g)|=+\infty$. Therefore, $a,c>0$ and $b>-\sqrt{ac}$. Conversely, if
$a,c>0$ and $b>-\sqrt{ac}$, then $G_1(g)$ is bounded, and therefore $|G_1(g)|<+\infty$.

Consequently, $g\in\F_4(\R^2)$ if and only if written in its canonical form $a,c>0$ and $b>-\sqrt{ac}$ and then $\F_4(\R^2)$ is open.

%In that case, $g\in\F_4(\R^2)$ if and only if $G_1(g)$ is bounded if and only if $a,c>0$ and $b>-\sqrt{ac}$.
%These last inequalities imply $\F_4(\R^2)$ is open.

%    In fact, suppose $a,c>0$ and $b>-\sqrt{ac}$. Then $h=b+\sqrt{ac}>0$ and then
%    $$
%    g(x,y)=(\sqrt{a}x^2-\sqrt{c}y^2)^2+2hx^2y^2.
%    $$
%If $(x,y)\in G_1(g)$, then $|\sqrt{a}x^2-\sqrt{c}y^2|\le1$, and $2hx^2y^2\le1$. Both inequalities imply $|x|,|y|$ are bounded. So $G_1(g)$ is bounded.%
%
%Clearly, if $G_1(g)$ is bounded, then $|G_1(g)|<+\infty$.

%Finally, suppose $|G_1(g)|<+\infty$. Then $a,c\ge0$ (otherwise $g(1,0)<0$ or $g(0,1)<0$). If $a=0$ or $c=0$, then $\iint e^{-g(x,y)}\,dxdy=+\infty$. If $a,c>0$ and $b\le -\sqrt{ac}$, then %$g(x,y)\le(\sqrt{a}x^2-\sqrt{c}y^2)^2=(\sqrt{a}x+\sqrt{c}y)^2(\sqrt{a}x-\sqrt{c}y)^2$. The change of variables $u=\sqrt{a}x+\sqrt{c}y$, $v=\sqrt{a}x-\sqrt{c}y$ and the fact $\iint e^{-u^2v^2}\,dudv=+\infty$ show that %$|G_1(g)|=+\infty$. Therefore, $a,c>0$ and $b>-\sqrt{ac}$.
%\textcolor{red}{Tambi\'en lo he retocado bastante. Repasar.}
\item Let $g(x,y,z)=x^4+y^4+z^4-2\sqrt{2}x^2yz$. Then $G_1(g)$ is unbounded (it contains the lines $y=z$, $x=\pm\root 4\of{2}y$). But $|G_1(g)|<+\infty$. In fact, using Lemma \ref{lem:IntegralOfg}, and since $g$ is even with respect to $x$ and for $y,z>0$, we have that $g(x,y,z)=g(x,-y,-z)\le g(x,-y,z)=g(x,y,-z)$, it is enough to prove that
    $$\iiint_{[0,+\infty)^3}
     e^{-g(x,y,z)}\,dxdydz<+\infty.$$

    The change of variables $x=u, y=u v, z=u w$, with Jacobian $J(u,v,w)=u^2$, rewrites the previous integral as
    $$\iiint_{[0,+\infty)^3} u^2 e^{-u^4 h(v,w)}\,dudvdw
   $$
    where $h(v,w)=1+v^4+w^4-2\sqrt{2}vw$. The change of variables
    $\overline{u}=u h(v,w)^\frac14$ shows that the previous integral equals
    $$
    \iint_{[0,+\infty)^2}\frac{dv\,dw}{ h(v,w)^{3/4}}\int_0^{+\infty} \overline{u}^2 e^{-\overline{u}^4}\,d\overline{u}.
    $$
    Therefore, it suffices to see that $ h(v,w)^{-3/4} $ is integrable in $[0,+\infty)^2$. Note that $h$ can be written as
    $$
    h(v,w)=\sqrt{2}(v-w)^2+(v^2-\gamma^2)^2+(w^2-\gamma^2)^2
    $$
    with $\gamma=2^{-1/4}$. Notice that $h(\gamma,\gamma)=0$ and that for every $(v,w)\in[0,+\infty)^2$ such that $(v,w)\neq (\gamma,\gamma)$ we have that $h(v,w)>0$.

    First, for $(v,w)\in[0,2\gamma]\times[0,2\gamma]$, the bound
    $$
    h(v,w)\ge \gamma^2((v-\gamma)^2+(w-\gamma)^2)\ge \sqrt{2}|v-\gamma| |w-\gamma|
    $$
    shows the integrability of $ h(v,w)^{-3/4} $ in $[0,2\gamma]\times [0,2\gamma]$.

    Second, for $(v,w)\in[0,2\gamma]\times[2\gamma,+\infty)$, the bound
    $$
    h(v,w)\ge \gamma^2(v-\gamma)^2+\tfrac{1}{4}w^4\ge \tfrac{1}{2^{1/4}}|v-\gamma| w^2
    $$
    shows the integrability of $ h(v,w)^{-3/4} $ in $[0,2\gamma]\times[2\gamma,+\infty)$. A similar bound shows the integrability in $[2\gamma,+\infty)\times[0,2\gamma]$.

    Finally, for $(v,w)\in[2\gamma,+\infty)\times[2\gamma,+\infty)$, the bound
    $$
    h(v,w)\ge \tfrac{1}{4}(v^4+w^4)\ge \tfrac{1}{2}v^2w^2
    $$
    shows the integrability of $ h(v,w)^{-3/4} $ in $[2\gamma,+\infty)\times[2\gamma,+\infty)$.

    \begin{figure}[!h]
    \includegraphics[scale=.35,bb=10 30 350 270,clip]{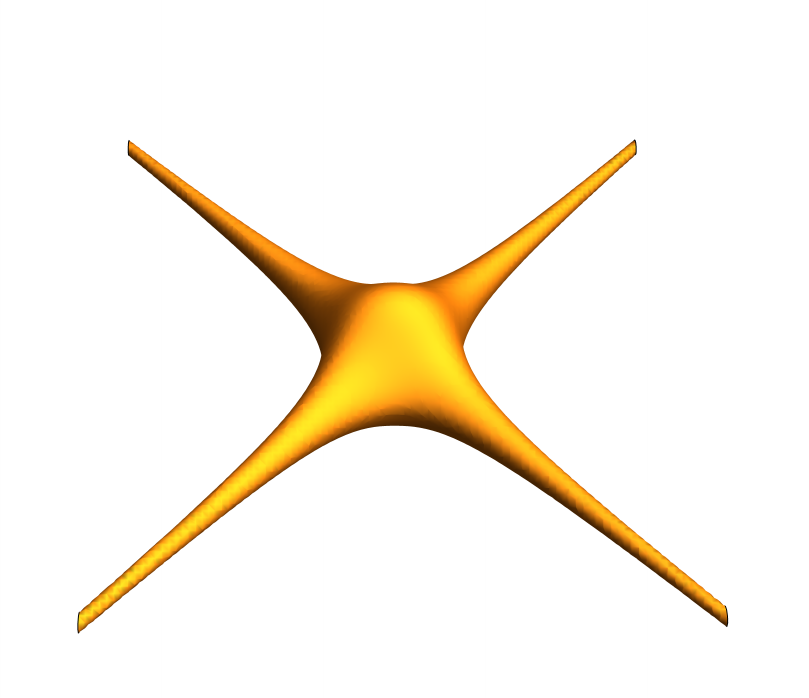}
    \caption{$G_1(g)$ for $g(x,y,z)=x^4+y^4+z^4-2\sqrt{2}x^2yz$.}
    \end{figure}

     Adding $x_i^4$ for the rest of variables, we can construct an example in $\R^n$ for $n\ge3$
     of a polynomial $g$ such that $G_1(g)$ is unbounded but  $|G_1(g)|<+\infty$.

    The polynomial $g_t(x,y,z)=x^4+y^4+z^4-tx^2yz$ for any $t>2\sqrt{2}$ does not belong to $\F_4(\R^n)$ (since $g_t(\root 4\of{y},y,y)=(4-t\sqrt{2})y^4<0$). A similar example can be constructed in $\R^n$ for $n\ge3$ as well, so $\F_4(\R^n)$ is not open for $n\ge3$.

\item Consider $g(x,y)=(x^2-y^2)^2(x^{d-4}+y^{d-4})$ in $\R^2$. Then, $g\in \F_d(\R^2)$, since $\{(x,y)\in\R^2:x^{d-4}+y^{d-4}\le x^2+y^2\}$ is compact and  $(x^2-y^2)^2(x^{2}+y^{2})\in\F_6(\R^2)$, by  Lemma \ref{lem:IntegralOfg} with $m=3$, and integrating in polar coordinates:
\[
\begin{split}
\int_{\R^2} & \exp\{-((x^2-y^2)^2(x^{2}+y^{2}))^{1/3}\}\,dxdy \\
& = \int_{0}^{2\pi}\int_0^{+\infty}r\exp\{-r^2(\cos^2\theta-\sin^2\theta)^{2/3}\} drd\theta \\
& = \int_0^{2\pi}\frac{d\theta}{2(\cos^2\theta-\sin^2\theta)^{2/3}} 
 %= \int_0^1\frac{dt}{t^{2/3}\sqrt{1-t^2}} = \frac{3\sqrt{\pi}\Gamma(7/6)}{\Gamma(2/3)}
<+\infty.
\end{split}
\]
%where we have used above polar coordinates% and a change of variables.

However, $G_1(g)$ is unbounded, since it contains the lines $y=\pm x$. Moreover, $g(x,y)-(1-t)x^d$ is not in $\F_d(\R^2)$ for any $t<1$, since it takes negative values  for $x=y$. Thus  $\F_d(\R^2)$ is not open.
    \begin{figure}[!h]
    \includegraphics[scale=.3,bb=20 20 350 350,clip]{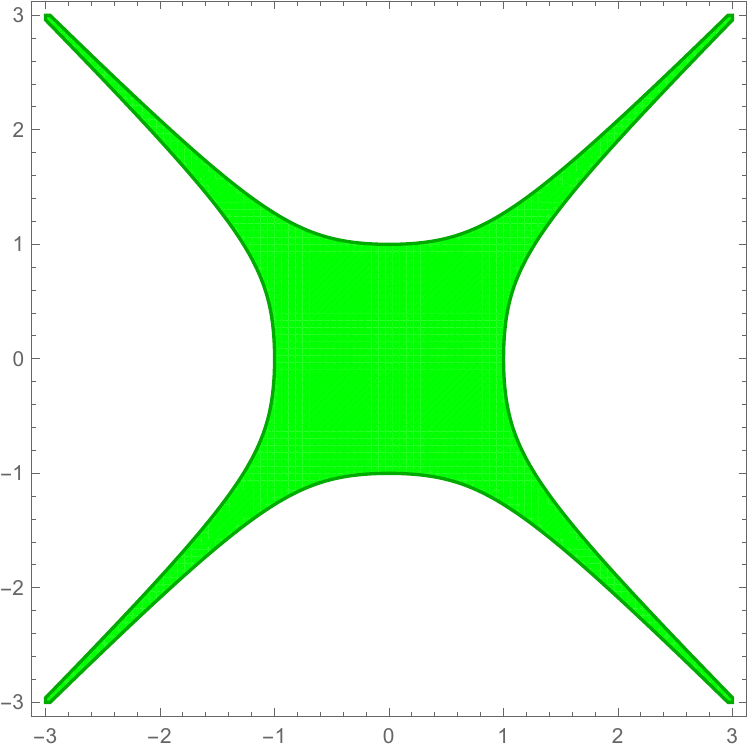}
    \caption{$G_1(g)$ for $g(x,y)=(x^2-y^2)^2(x^{2}+y^{2})$.}
    \end{figure}

    Adding $x_i^d$ for the rest of variables, we can construct an example in $\R^n$ for $n\ge3$.\qedhere
\end{enumerate}
\end{proof}

\begin{rmk}
It is worth mentioning here the connection of homogeneous positive multivariate polynomials with Hilbert's seventeenth problem, one of the 23 Hilbert problems set out in a celebrated list compiled in 1900 by David Hilbert.
It concerns the expression of positive definite rational functions as sums of quotients of squares.

In 1888, Hilbert himself \cite{Hi} showed that every non-negative homogeneous polynomial in $n$ variables and degree $d$ can be represented as sum of squares of other polynomials if and only if either (1) $ n = 2$, (2) $d = 2$ or (3) $n = 3$ and $d = 4$. This result, however, cannot be used in the previous lemma, in the study of the structure of polynomials, since that canonical way of writing homogeneous polynomials is not specific enough to  suggest
 a suitable change of variables,
  as it was done in (2) or (3) in Lemma \ref{F_d}.
\end{rmk}

The following result, of independent interest, will be needed for the study of the convergence of coefficients of polynomials.

\begin{propo}\label{prop:injective}
Consider the map
$
\Phi:\F_d(\R^n)\to\R^{h_d(n)}
$
given by
$$
\Phi(g)=\left(\tfrac{1}{\Gamma\left(\tfrac{n }{d}+1\right)}\int_{\R^n} x^\alpha \exp(-g(x))\,dx\right)_{\alpha\in\N_d^n}.
$$

The map $\Phi$ is  one-to-one, continuous and differentiable, and its inverse (defined on the image set)
is also continuous and differentiable.
\end{propo}

\begin{proof} The integrability is guaranteed by Lemma \ref{identity_r}, so $\Phi$ is well defined on $\F_d(\R^n)$.  As it was shown in \cite{lasserre}, the function
$$
\omega(g)=|G_1(g)|=
\frac{1}{\Gamma\left(\frac{n }{d}+1\right)}\int_{\R^n}\exp(-g(x))\,dx,\hskip1cm g\in\F_d(\R^n)
$$
is a strictly convex function. Moreover, its gradient is $\nabla \omega=-\Phi$. Consequently, its Hessian, a positive semi-definite matrix, is the Jacobian matrix of $-\Phi$. More precisely, as it was shown in \cite{lasserre},
\begin{equation}\label{eq:DerivativeOfPsi}
\frac{\partial \Phi_\alpha}{\partial g_\beta}=
-\tfrac{1}{\Gamma\left(\tfrac{n }{d}+1\right)}
\int_{\R^n} x^{\alpha+\beta}\exp(-g(x))\,dx,
\end{equation}
where $\beta\in\N_d^n$, and thus for every $(h_\alpha)\in\R^{h_d(n)}$,
$$
\sum_{\alpha,\beta\in\N_d^n} h_\alpha h_\beta\frac{\partial \Phi_\alpha}{\partial g_\beta}=-\tfrac{1}{\Gamma\left(\tfrac{n }{d}+1\right)}
\int_{\R^n} h(x)^2\exp(-g(x))\,dx,
$$
where $h=\sum_{\alpha\in\N_d^n}h_\alpha x^\alpha$. This shows that the matrix $(\frac{\partial \Phi_\alpha}{\partial g_\beta})$ is negative semi-definite. By \cite[Theorem 6]{GN} we get that $\Phi$ is globally one-to-one. The continuity and differentiability of the inverse function follow from the Inverse Function Theorem.
\end{proof}

%\begin{example}
%For $n=d=2$, the function
%$$
%\Phi(g_{(2,0)},g_{(1,1)},g_{(0,2)})=\left(
%\tfrac{4\pi g_{(0,2)}}{(4g_{(2,0)}g_{(0,2)}- g_{(1,1)}^2)^{3/2}},
%\tfrac{-2\pi g_{(1,1)}}{(4g_{(2,0)}g_{(0,2)}- g_{(1,1)}^2)^{3/2}},
%\tfrac{4\pi g_{(2,0)}}{(4g_{(2,0)}g_{(0,2)}- g_{(1,1)}^2)^{3/2}}
%\right)
%$$
%is bijective from $\{(g_{(2,0)},g_{(1,1)},g_{(0,2)})\in\R^3:4 g_{(2,0)}g_{(0,2)}-g_{(1,1)}^2>0\}$ onto %$\{(h_{(2,0)},h_{(1,1)},h_{(0,2)})\in\R^3:h_{(2,0)}h_{(0,2)}-h_{(1,1)}^2>0\}$.
%\end{example}

\begin{example}
For $n=d=2$, the function $\Phi:\F_2(\R^2)\to\R^3$ is defined as
$$
\Phi(g)=(4ac- b^2)^{-3/2}
\left(
{4\pi c},
{-2\pi b},
{4\pi a}
\right)
$$
is bijective from $\{g\in\F_2(\R^2):g(x,y)=ax^2+bxy+cy^2\,(a,c>0,\,4 ac>b^2)\}$ onto $\{(a',b',c')\in\R^3:
a',c'>0,\,
a'c'>(b')^2\}$.
\end{example}

\section{Aproximation of log-concave functions by polynomials}\label{section:prob1}

Before showing the existence of a solution for \textit{Problem \ref{problem1}} whenever $f\in\mathcal{F}(\R^n)$ we will start by considering the following example, which shows that without any convexity assumption on $f$ \textit{Problem \ref{problem1}} can be ill-posed (see also Example \ref{example:quasiconcave}). Nevertheless, if $f:\R^n\to[0,+\infty)$ is a function with $\Vert f\Vert_\infty=f(0)=1$ and compact support, considering $K$ the convex hull of $\textrm{supp}f$ we have that $\chi_K\in\mathcal{F}(\R^n)$ and $f\leq\chi_K$. Therefore, a solution of \textit{Problem \ref{problem1}} for $\chi_K$ will provide a function $g\in\H_d[x]$ and a $t\geq 1$ for which $f(x)\le t \exp(-g(x)^{1/d})$ and $|G_1(g)|$ is finite.

%Let us start solving the minimization \textit{Problem \ref{problem1}}.
%Before it, we prove in the example below that this problem cannot be posed for general integrable %functions
%(see also Example \ref{example:quasiconcave}).

\begin{example}\label{example:no_integ}
Let $f=\chi_A$, where $A$ is the union of concentric spherical shells:
$$
A=
\bigcup_{k=1}^\infty
\{x\in\R^n:k\le |x|\le k+\tfrac{1}{2^k}\}.
$$
The function $f$ is integrable, but  $f(x)\le t \exp(-g(x)^{1/d})$ for some $g\in\H_d[x]$  would imply $g(x)\le (\log t)^d$ for all $x\in A$. The only $d$-homogeneous bounded polynomial is $g=0$, for which $|G_1(g)|=+\infty.$
\end{example}

Having this example in mind, we will solve \textit{Problem \ref{problem1}} in the setting of
 log-concave integrable functions.

For $f\in\mathcal F(\mathbb R^n)$ with $\Vert f\Vert_\infty=f(0)=1$, let $K_\lambda(f)=\{x\in\mathbb R^n:f(x)\geq\lambda\}$. These super-level sets are convex by the log-concavity of $f$.

Given any quasi-convex function $f$ (i.e., a function whose super-level sets are convex) and $t\geq1$, let
$$
H_t(f)
=\bigcup_{\lambda\in(0,1)}\log(t/\lambda)^{-1}K_\lambda(f).
$$
Since $K_\lambda(f)$ are convex sets  that contain the origin, $H_t(f)$ is decreasing in $t\geq1$.

The following  lemma shows the relation between $H_t(f)$ and \textit{Problem \ref{problem1}}.

\begin{lemma}\label{lem:equiv}
Let $f:\R^n\to[0,+\infty)$ be a log-concave function with $\Vert f\Vert_\infty=f(0)=1$, $d\in\mathbb N$ even $g\in\mathbb H_d(\mathbb R^n)$ and $t\geq1$.
The following conditions are equivalent:
\begin{enumerate}[(i)]
\item $f(x) \le t\exp(-g(x)^{1/d})$ for all $x\in\R^n$.
\item $H_t(f) \subset G_1(g).$
\end{enumerate}
\end{lemma}

\begin{proof}
Assume that condition \textit{(i)} holds. Then, for any $x\in H_t(f)$, there exists $\lambda\in(0,1)$ such that $(\log(t/\lambda))x\in K_\lambda(f)$. In other words,
$f((\log(t/\lambda)x)\ge\lambda
$. But then, using \textit{(i)} and the homogeneity of $g$,
$$
\lambda\le f((\log(t/\lambda)x))
\le
t \exp(-g(\log(t/\lambda)x)^{1/d})
=
 t \exp(-\log(t/\lambda)g(x)^{1/d})
$$
and hence $g(x)\le1$. So  \textit{(ii)} is proved.

Conversely, assume that \textit{(ii)} holds and take any $x\in\R^n$.

If $f(x)<1$, let $\lambda=f(x)$. Clearly, $(\log(t/\lambda))^{-1}x\in H_t(f)$, so using condition \textit{(ii)},
$
(\log(t/\lambda))^{-1}x\in G_1(g),
$
and therefore, $f(x)\le t \exp(-g(x)^{1/d})$.

Now assume that $f(x)=1$. If $t>1$, then $1\le t \exp(-g(x)^{1/d})$ is equivalent to $g(x)\le (\log t)^d$. Take any $\lambda\in(0,1)$. Then $x\in K_\lambda(f)$ and therefore
$$
((\log(t/\lambda))^{-1}x\in
((\log(t/\lambda))^{-1} K_\lambda(f)
\subset
H_t(f)
\subset G_1(g)
$$
by condition \textit{(ii)}. Consequently, $g(x)\le (\log(t/\lambda))^d$ for any $\lambda\in(0,1)$. But then $g(x)\le (\log t)^d$.

Finally, assume that $f(x)=t=1$. It is left to show that $g(x)=0$. Since $H_t(f)$ is decreasing in $t\ge1$,
$$
H_t(f)\subset H_1(f)\subset G_1(g)
$$
by condition \textit{(ii)}. Using the case $t>1$ proved above, $g(x)\le (\log t)^d$. This inequality is true for any $t>1$, and then $g(x)=0$.
\end{proof}

The following result  gives a monotonicity  behaviour, crucial in the study of the  minimization problem.
Moreover, it will imply some consequences about  the boundedness of $H_t(f)$.

\begin{lemma}\label{lem:reverseH_tContainment}
Let $f:\R^n\to[0,+\infty)$ be a log-concave function with $\Vert f\Vert_\infty=f(0)=1$. Then, for every $1<t_0<t_1$, we have that
\begin{equation*}
(\log t_0 )H_{t_0}(f)\subset (\log t_1) H_{t_1}(f).
\label{inclusion_log}
\end{equation*}
\end{lemma}

\begin{proof}
Since $\log$ is an increasing function and for every $\lambda\in(0,1)$ $K_\lambda(f)$ is star-shaped with respect to the origin, we have that
\begin{align*}
H_{t_0}(f)
&=
\frac{\log t_1}{\log t_0}\bigcup_{\lambda\in(0,1)}\frac{1}{\log t_1+\frac{\log t_1}{\log t_0}\log(1/\lambda)}K_\lambda(f)
\\
&\subset
\frac{\log t_1}{\log t_0}\bigcup_{\lambda\in(0,1)}\frac{1}{\log t_1+\log(1/\lambda)}K_\lambda(f)
\\
&=
\frac{\log t_1}{\log t_0}H_{t_1}(f).\qedhere
\end{align*}
\end{proof}

\begin{rmk} In the previous two lemmas we have not assumed the integrability of $f$. Moreover, in Lemma \ref{lem:equiv} we have only used  the log-concavity of $f$ in the case $t=1$; indeed, in both lemmas, the only fact needed is that $\alpha_1 K_\lambda(f)\subset \alpha_2 K_\lambda(f)$ for any $0<\alpha_1\le\alpha_2$, which is equivalent to the fact that $K_\lambda(f)$ is star-shaped with respect to the origin.
Finally, the inclusions in Lemma \ref{lem:reverseH_tContainment} above are sharp, see for instance Example \ref{examples1} \eqref{ex:infty_at_1}.
\end{rmk}

Now we can state the boundedness of $H_t(f)$.

\begin{lemma}\label{lem:boundnessH_t}
Let $f\in\mathcal F(\R^n)$ with $\Vert f\Vert_\infty=f(0)=1$ and an even $d\in\mathbb N$. Then $H_t(f)$ is bounded for every $t>1$. Moreover, if $H_1(f)$ is unbounded, then $|H_1(f)|=+\infty$.
\end{lemma}

\begin{proof}
Let $t_0\exp(-\Vert\cdot\Vert_\elip)$ be the unique minimization ellipsoid verifying \eqref{ineq:func1} for the even log-concave function $f_s=\exp(-u_s)$, being $u_s$ the convex function whose epigraph is the convex hull of the functions $u(x)$ and $u_-(x):=u(-x)$. That is,
$$
\mathrm{epi}(u_s)=\mathrm{conv}(\mathrm{epi}(u),\mathrm{epi}(u_-)).
$$
 Notice that $f\le f_s$ and, since $f\in\mathcal{F}(\R^n)$, then also $f_s\in\mathcal{F}(\R^n)$ (see \cite{AAGJV}).

Let us observe that $g=\Vert\cdot\Vert_\elip^d\in\mathbb H_d(\R^n)$. Then $f\leq t_0\exp(-g^{1/d})$, which by
Lemma \ref{lem:equiv} means that $H_{t_0}(f)\subset G_1(g)$. In this case $G_1(g)$ is an ellipsoid, hence bounded,
so $H_{t_0}(f)$ is bounded too. Since $H_t(f)$ is decreasing on $t\geq 1$ and
by Lemma \ref{lem:reverseH_tContainment}, $(\log t)H_t(f)$ is increasing on $t>1$, we have that $H_t(f)$
is bounded for every $t>1$.

Finally, let us assume that $H_1(f)$ is unbounded. Let us observe that
$$
\frac{1}{\log(1/\lambda_1)}K_{\lambda_1}(f)\subset\frac{1}{\log(1/\lambda_2)}K_{\lambda_2}(f)
$$
for any $0<\lambda_1<\lambda_2\leq 1$. Indeed, take $x\in K_{\lambda_1}(f)$. Since $f(0)=1$ and $\log(1/\lambda_2)/\log(1/\lambda_1)\in[0,1]$,
$$
f\left(\frac{\log(1/\lambda_2)}{\log(1/\lambda_1)}x\right)\geq f(x)^{\frac{\log(1/\lambda_2)}{\log(1/\lambda_1)}}\geq \lambda_1^{\frac{\log(1/\lambda_2)}{\log(1/\lambda_1)}}=\lambda_2.
$$
Therefore $H_1(f)=\bigcup_{\lambda\in(0,1)}(\log (1/\lambda))^{-1}K_{\lambda}(f)$ is an increasing union of
convex sets, i.e., convex.
We can assume that $\int_{\R^n} f(x)dx>0$, and hence $|K_\lambda(f)|>0$ for some $\lambda\in(0,1)$, so it has non empty interior, and so has $H_1(f)$.
Since $H_1(f)$ is unbounded (and convex), then $|H_1(f)|=+\infty$, as desired.
\end{proof}

\begin{example}\label{example:quasiconcave}
The previous result is not true if the log-concavity assumption on $f$ is dropped. %even if for instance, in the case of some integrable
%\emph{quasiconcave functions} (i.e., functions with all level sets being convex sets).
For instance, given $K\subseteq\R^n$ a convex body with $0\in K$ and $\alpha>n$ consider
$$
f(x)=
\begin{cases} 1 & x\in K,
\\
\Vert x\Vert_K^{-\alpha} & \text{otherwise.}
\end{cases}
$$
Then $f$ is quasi-concave and

\begin{eqnarray*}
\int_{\R^n}f(x)dx & = &|K|+\int_{\R^n\setminus K}\Vert x\Vert_K^{-\alpha}dx\cr
& = &|K|+\int_0^1|\{x\in\R^n\setminus K:\Vert x\Vert_K^{-\alpha}\geq t\}|dt\cr
& = &|K|+\int_0^1|t^{-1/\alpha}K\setminus K|dt=\frac{\alpha}{\alpha-n}|K|<+\infty.
\end{eqnarray*}

Its super-level sets are
$$
K_\lambda(f)=
\begin{cases}
\lambda^{{-1/\alpha}}K & 0<\lambda<1
\\
K &  \lambda=1,\end{cases}
$$
(note that these sets are convex), and thus for every $t\ge1$
$$
H_t(f)=
\bigcup_{\lambda\in(0,1)}\frac{1}{\lambda^{1/\alpha}\log (t/\lambda)}K=\R^n.\qedhere
$$
\end{example}

The following result shows the relation between $H_t(f)$ and the super-level sets of $f^\circ$. Recall that the polar body of a convex body $K$ containing the origin is $K^\circ=\{x\in\R^n\,:\,\langle x,y\rangle\leq 1 \text{ for every } y\in K\}$.

\begin{lemma}\label{lem:polar}
Let $f\in\mathcal{F}(\R^n)$ with $\Vert f\Vert_\infty=f(0)=1$. Then for any $t>1$ $H_t(f)$ is convex  and
$$
\overline{H_t(f)}=\{x\in\R^n\,:\,f^\circ(x)\geq\tfrac{1}{t}\}^\circ=(K_{\frac{1}{t}}(f^\circ))^\circ.
$$
\end{lemma}

\begin{proof} %OJO AL CAMBIO DE DEFINICIÓN DE $H_t$

Calling $t=e^r$ for some $r>0$ and $\lambda=e^{-s}$ for $s\geq0$, we have that
$$
H_{e^{r}}(f)=\bigcup_{s\geq 0}\frac{K_{e^{-s}}(f)}{r+s}.
$$

We start showing that $H_{e^r}(f)$ is convex.
Let $x_1,x_2\in H_t(f)$ and $0\leq \theta\leq 1$. Then, there exist $s_1,s_2\geq 0$ such that $y_1=(r+s_1)x_1\in K_{e^{-s_1}}(f)$ and $y_2=(r+s_2)x_2\in K_{e^{-s_2}}(f)$. Letting $0\leq\lambda=\frac{\theta(r+s_1)}{(1-\theta)s_2+\theta s_1+r}\leq 1$ and $s_\theta=(1-\lambda)s_1+\lambda s_2$ we have that
$$
(r+s_\theta)[(1-\theta)x_1+\theta x_2]=(1-\lambda)y_1+\lambda y_2
$$
and then, since
$$
f((1-\lambda)y_1+\lambda y_2)\geq e^{-[(1-\lambda)s_1+\lambda s_2]}=e^{-s_\theta},
$$
we have that $(r+s_\theta)[(1-\theta)x_1+\theta x_2]\in K_{e^{-s_\theta}}$ or, equivalently, $(1-\theta)x_1+\theta x_2\in\frac{K_{e^{-s_\theta}}}{r+s_\theta}\subseteq H_{e^r}(f)$.

%If $r>1$ we have a geometric interpretation of $H_{e^r}(f)$- Notice that for every $s\geq0$, $x\in \frac{K_{e^{-s}}(f)}{r+s}$ if and only if $(r+s)x\in K_{e^{-s}}(f)$, which happens if and only if $(x, -r+1)\in\textrm{conv}\{K_{e^{-s}}\times\{s\},(0,-r)\}$. Therefore, if $f(x)=e^{-u(x)}$ where $u:\R^n\to[0,+\infty]$ is a convex function, we have that
%$$
%\bigcup_{s\geq 0}\frac{K_{e^{-s}}(f)}{r+s}\times\{-r+1\}=\textrm{conv}\{\textrm{epi}(u),(0,-r)\}\cap(\R^n\times\{-r+1\})
%$$
%and then $H_{e^{r}}(f)$ is convex.
Let $u:\R^n\to[0,+\infty]$ be the convex function such that $f=\exp(-u)$. Note that
$$
f^\circ(x)\geq\frac{1}{t}\Leftrightarrow u^*(x)\leq r,
$$
where $u^*(x)$ is the Legendre transform of $u$
$$
u^*(x)=\sup_{y\in\R^n}(\langle x,y\rangle-u(y)).
$$
Therefore $u^*(x)\leq r$ if and only if for every $y\in\R^n$
$$
\langle x,y\rangle\leq u(y)+r,
$$
which happens if and only if for every $s\geq 0$ and every $y\in\{y\in\R^n\,:\,u(y)\leq s\}=K_{e^{-s}}(f)$ we have
%$y\in\{y\in\R^n\,:\,u(y)=s\}=\partial K_{e^{-s}}$, we have
$$
\langle x,y\rangle\leq s+r,
$$
which is equivalent to the fact that for every $s\geq0$, $x\in (s+r)K_{e^{-s}}^\circ$. Thus,
$$
K_{\frac{1}{t}}(f^\circ)=\{x\in\R^n\,:\,f^\circ(x)\geq\frac{1}{t}\}=\bigcap_{s\geq 0}(s+r)K_{e^{-s}}^\circ=\bigcap_{s\geq 0}\left(\frac{K_{e^{-s}}}{s+r}\right)^\circ
$$
and then
$$
(K_{\frac{1}{t}}(f^\circ))^\circ=\left(\bigcap_{s\geq 0}\left(\frac{K_{e^{-s}}}{s+r}\right)^\circ\right)^\circ=\overline{\textrm{conv}\left(\bigcup_{s\geq 0}\frac{K_{e^{-s}}}{s+r}\right)}=\overline{H_t(f)}.\hfill\qedhere
$$
\end{proof}

Using Lemmas \ref{lem:equiv} and \ref{lem:IntegralOfg} we can reformulate \textit{Problem \ref{problem1}} as follows: Given $f\in\mathcal{F}(\R^n)$ with $\Vert f\Vert_\infty=f(0)=1$, find $t_0\ge1$ and $g_0\in \H_d[x]$ such that
$H_{t_0}(f)\subset G_1(g_0)$ and
\[
t_0 |G_1(g_0)|
=
\inf t|G_1(g)|=\inf_{t\geq 1}\left(t\inf_{g\in\mathbb H_d(\mathbb R^n)}|G_1(g)|\right)
\]
where the infimum is taken among all $(t,g)$ verifying $t\ge1$, $g\in\H_d[x]$, and $H_t(f)\subset G_1(g)$.

Let us observe that for every $t> 1$,
 the infimum above over $g\in\H_d(\R^n)$ such that $H_t(f)\subset G_1(g)$  is a minimum.
 Indeed,  we may apply the minimization problem solved by Lasserre to $\overline{H_t(f)}$, the closure of ${H_t(f)}$, (which by Lemma \ref{lem:boundnessH_t} is compact), and get  $g_t\in\H_d(\R^n)$  the only polynomial  verifying
$$
H_t(f)\subset G_1(g_t)
$$
with minimum volume $|G_1(g_t)|$ among all $g\in\H_d(\R^n)$ such that $H_t(f)\subset G_1(g)$.

We can also apply the argument  to $H_1(f)$ (or its closure) if $H_1(f)$ is bounded (if it is unbounded, then
Lemma \ref{lem:boundnessH_t} implies that $|H_1(f)|=+\infty$, and hence it does not play any role in the minimization
problem). Let $g_1$ be the corresponding
polynomial to $H_1(f)$ (if it is bounded). Then, the infimum in \textit{Problem \ref{problem1}} can be rewritten as
\[
\inf_{t\geq 1}t|G_1(g_t)|.
\]
For any $t\ge1$, let $v(t)=|G_1(g_t)|$ (consider $v(1)=+\infty$ if $H_1(f)$ is unbounded) and $\phi(t)=tv(t)$ be the function to be minimized. Some properties for these functions are needed to solve the problem. Let us start with a technical lemma.

\begin{lemma}\label{lem:IneqLogConvex}
Let $t_0,t_1,d\geq 1$, $\theta\in[0,1]$, and $a\in(0,1]$. Then
$$
(1-\theta)\left(\log\tfrac{t_0}{a}\right)^d
+
\theta\left(\log\tfrac{t_1}{a}\right)^d
\leq
\left(\log\tfrac{t_\theta}{a}\right)^d,
$$
where $t_\theta$ is defined by the identity $(\log t_\theta)^d=(1-\theta)(\log t_0)^d+\theta (\log t_1)^d$.
\end{lemma}

\begin{proof}
If $d=1$ the inequality in the statement is trivially an equality. Assume $d>1$. The inequality above can be reformulated as
\[
F(a)=a\exp\left[(1-\theta)\left(\log\tfrac{t_0}{a}\right)^d+\theta\left(\log\tfrac{t_1}{a}\right)^d\right]^{\frac1d}\leq t_\theta
\]
for any $a\in(0,1]$. Since $F(1)=t_{\theta}$, it is enough to prove that $F$ is increasing on $(0,1]$.
Indeed,
considering the change of variables $a=e^{-b}$, $t_i=e^{s_i}$, $i=0,1$,  $F$ is increasing if and only if $G$ is decreasing on $[0,+\infty)$, where
\[
G(b)=\log F(e^{-b})=
\left((1-\theta)(s_0+b)^d+\theta(s_1+b)^d\right)^\frac{1}{d}-b,
\hskip1cm
b\in[0,+\infty)
\]
for  any $s_0,s_1\geq 0$.
Its derivative equals
$$
G'(b)=
\left[\frac{\left((1-\theta)(s_0+b)^{d-1}+
\theta(s_1+b)^{d-1}\right)^{\frac{1}{d-1}}}
{\left((1-\theta)(s_0+b)^{d}+\theta(s_1+b)^{d}\right)^{\frac{1}{d}}}\right]^{d-1}-1.
$$
Letting $u_i=(s_i+b)^{d-1}$, $i=0,1$, then $G'(b)\leq 0$ rewrites as
\[
\left((1-\theta)u_0+\theta u_1\right)^{\frac{d}{d-1}} \leq (1-\theta)u_0^{\frac{d}{d-1}}+\theta u_1^{\frac{d}{d-1}},
\]
which is a consequence of the convexity of $u\mapsto u^{\frac{d}{d-1}}$.
\end{proof}

\begin{lemma}\label{lem:prop_of_v}
Let $f\in\mathcal F(\R^n)$ with $\Vert f\Vert_\infty=f(0)=1$ and $d\in\mathbb N$ even.
Then $v(t)$ is a decreasing function and $(\log t)^n v(t)$ is increasing in $t$.

Moreover, if $t_0,t_1>1$ and $\theta\in[0,1]$, then
$$
v(t_\theta)\le v(t_0)^{1-\theta}v(t_1)^\theta,
$$
where $(\log t_\theta)^d=(1-\theta)(\log t_0)^d+\theta (\log t_1)^d$.
\end{lemma}

\begin{proof}
$v$ is decreasing by definition.
On the other hand, taking volumes in the inclusion given in Lemma \ref{lem:reverseH_tContainment}, for any $1<t_0< t_1$
$$
(\log t_0)^n  v(t_0)\le(\log t_1)^n v(t_1).
$$
Now, for any $i=0,1$, and for every $x\in\R^n$,
$$
f(x)\le t_i \exp(-g_{t_i}(x)^{1/d}).
$$
Then, if $f(x)\not=0$,
$$
g_{t_i}(x)\le \left(\log\tfrac{t_i}{f(x)}\right)^d
$$
for $i=0,1$. Since $f(x)\in(0,1]$ and $t_i\geq 1$, then by Lemma \ref{lem:IneqLogConvex} we get that
\begin{equation*}
(1-\theta)g_{t_0}(x)+\theta g_{t_1}(x) \le
(1-\theta)\left(\log\tfrac{t_0}{f(x)}\right)^{d}+\theta\left(\log\tfrac{t_1}{f(x)}\right)^d
\le
\left(\log\tfrac{t_\theta}{f(x)}\right)^d,
\end{equation*}
and  by Lemma \ref{lem:equiv}, $H_{t_\theta}(f)\subset G_1((1-\theta)g_{t_0}+\theta g_{t_1})$.
%%%%% Este therefore no tiene sentido, ya que eso que se dice es siempre verdad por definición
%Therefore,
%$$
%f(x)\le t_\theta  \exp(-g_{t_\theta}(x)^{1/d}).
%$$
Since $(1-\theta) g_{t_0}+ \theta  g_{t_1}\in \H_d(\R^n)$,
the minimization property of $g_{t_\theta}$ implies that
$ |G_1(g_{t_\theta})|\le |G_1((1-\theta) g_{t_0}+ \theta  g_{t_1})|$, and using Lemma \ref{lem:IntegralOfg} and H\"older's inequality

\begin{equation}\label{convexity}
\begin{aligned}
 |G_1(g_{t_\theta})|
 &\le
 |G_1((1-\theta) g_{t_0}+ \theta  g_{t_1})|
 \\
 &=\Gamma(\tfrac{n}{d}+1)^{-1}\int_{\R^n} \exp(-((1-\theta) g_{t_0}(x)+ \theta  g_{t_1}(x)))\,dx
 \\
 &\le
 \left(\Gamma(\tfrac{n}{d}+1)^{-1}\int_{\R^n} \exp(-g_{t_0}(x))\,dx\right)^{1-\theta}
 \left(\Gamma(\tfrac{n}{d}+1)^{-1}\int_{\R^n} \exp(-g_{t_1}(x))\,dx\right)^{\theta}
 \\
 &=
 |G_1(g_{t_0})|^{1-\theta}|G_1(g_{t_1})|^{\theta}.\qedhere
\end{aligned}
\end{equation}
\end{proof}

Now we can prove the first main result.
\begin{proof}[Proof of Theorem \ref{thm:prob1}]
By Lemma  \ref{lem:prop_of_v}, the function $s\in (0,+\infty)\mapsto \log v\left(e^{s^{1/d}}\right)$ is a convex function. This implies that $v$ is a continuous function on $(1,+\infty)$.
We will also prove that
\begin{equation}\label{continuity_1}
\lim_{t\to1^+}v(t)=v(1).
\end{equation}

Recall that $H_t(f)$ is decreasing in $t\geq 1$. That implies that there exists
$$
\lim_{t\to1^+}|H_t(f)|\le|H_1(f)|.
$$
% and $\overline{\bigcup_{t>1}H_t(f)}=\overline{H_1(f)}$.
%$\bigcup_{t>1}H_t(f)=H_1(f)$.
We have that   $\mu{H_1(f)}\subset{\bigcup_{t>1}H_t(f)}$ for any $\mu<1$. Indeed,  $x\in H_1(f)$ if and only if $f(\log(1/\lambda)x)\ge\lambda$ for some $\lambda\in(0,1)$. Take $t>1$ so that $\mu\log(t/\lambda)=\log(1/\lambda)$. Then $\mu x\in\log (t/\lambda)^{-1}K_\lambda(f)\subset H_t(f)$.

The Monotone Convergence Theorem ensures that
$\mu^n|H_1(f)|\le \lim_{t\rightarrow 1^+}|H_t(f)|$ (even if $|H_1(f)|=\infty$) for any $\mu<1$. Then
$$
\lim_{t\to1^+}|H_t(f)|=
|H_1(f)|.
$$

First, assume that $H_1(f)$ is bounded.
Using the minimization property for $|G_1(g_t)|$, we have that $(|G_1(g_t)|)_{t>1}$ is decreasing, so there exists $\lim_{t\to1^+}|G_1(g_t)|\le |G_1(g_1)|$.

Using \eqref{eq:set_touching} (and repeating some of the contact points and the coefficients if necessary), for any $t\ge1$, there are
$y_1^{(t)},\dots,y_{h_d(n)}^{(t)}\in H_t(f)$, $\lambda_1^{(t)},\dots,\lambda_{h_d(n)}^{(t)}\ge0$,
such that $g_t(y_i^{(t)})=1$ for $i=1,\dots,h_d(n)$, and
\begin{equation}\label{eq:contact1}
\int_{\mathbb R^n}x^\alpha e^{-g_t(x)}dx=\sum_{i=1}^{h_d(n)}\lambda_i^{(t)}(y_i^{(t)})^\alpha
\end{equation}
for every $\alpha\in\N_d^n$. Moreover, using the trace identity and Lemma \ref{identity_r},
$$
\frac{n}{d}\Gamma(\tfrac{n}{d}+1)|G_1(g_t)|=\int_{\mathbb R^n}g_t(x)e^{-g_t(x)}dx=\sum_{i=1}^{h_d(n)}\lambda_i^{(t)}.
$$
Using that $H_t(f)\subset H_1(f)$ and subsequently, $|G_1(g_t)|\le |G_1(g_1)|$, all coefficients $\lambda_i^{(t)}$ are uniformly bounded by
$$
0\le
\lambda_i^{(t)}\le
\frac{n}{d}\Gamma(\tfrac{n}{d}+1)|G_1(g_1)|
$$
and all the vectors $y_i^{(t)}$ lie in the same bounded set $H_1(f)$. These uniformly bounding conditions together with the set of equalities in \eqref{eq:contact1}, for every $\alpha\in\N^n_d$,
 %together with the uniform upper bound for the coefficients $\lambda_i^{(t)}$ and the fact that all the vectors $y_i^{(t)}$ lie in the same bounded set $H_1(f)$
 imply that there exists a compact set $\Omega\subset\R^{h_d(n)}$ such that for every $t\geq 1$, $\Phi(g_t)\in \Omega$, where $\Phi$ is the map defined in Proposition \ref{prop:injective}. Using Proposition \ref{prop:injective}, we have that the coefficients of all $g_t$ are uniformly bounded. Thus, taking a sequence $(t_k)$ converging to $1$, considering the sequence $g_{t_k}$, and passing to a convergent subsequence, we can construct a polynomial $g_0\in\H_d(\R^n)$ whose coefficients are the limit of the coefficients of such a subsequence of $(g_{t_k})$. Using that $H_t(f)\subset G_1(g_t)$ for any $t>1$ and Lemma \ref{lem:equiv}, and taking limit, we get that $H_1(f) \subset G_1(g_0)$. Using that the miminizing property defining $g_1$, we have $|G_1(g_1)|\le |G_1(g_0)|$, and using Lemma \ref{lem:IntegralOfg} and Fatou's lemma, $|G_1(g_0)|\le\lim_{t\to1^+}|G_1(g_t)|$ (since this limit exists).
%But, taking limit again, we have $|G_1(g_1)|\le\lim_{t\to1^+}|G_1(g_t)|$.
Then $|G_1(g_1)|\leq \lim_{t\to1^+}|G_1(g_t)|$ and
$$
\lim_{t\to1^+}|G_1(g_t)|=|G_1(g_1)|,
$$
as desired. Note that, using the equalities \eqref{eq:contact1}, and taking again a subsequence, we get the same equalities for $g_0$, for some coefficients and contact points in $H_1(f)$. Since these equalities characterize $g_1$, we have $g_0=g_1$.

If $H_1(f)$ is unbounded, then $|H_1(f)|=+\infty$ by Lemma \ref{lem:boundnessH_t}. The Monotone Convergence Theorem ensures again that
$\lim_{t\rightarrow 1+}|H_t(f)|=+\infty$. Since $H_t(f)\subset G_1(g_t)$,
$$
\lim_{t\rightarrow 1+}|G_1(g_t)|=+\infty
$$
and the proof of \eqref{continuity_1} is completed.

Using Lemma \ref{lem:prop_of_v}, the function $\phi(t)=\frac{t}{(\log t)^n}{(\log t)^n}v(t)$ is the product of two positive increasing functions in $[e^n,+\infty)$.
If $H_1(f)$ is bounded, $\phi$ attains its minimum in $[1,e^n]$ by continuity.
If $H_1(f)$ is unbounded, then $\lim_{t\to1^+}\phi(t)=+\infty$ and so, $\phi$ attains its minimum in $(1,e^n]$ by continuity.
In both cases, this is the minimum of $\phi$ in $[1,+\infty)$.
\end{proof}

%There are examples of $f\in\mathcal F(\R^n)$ for which $|H_1(f)|=\infty$ (see Examples \ref{ex:unbounded_1}).

The end of this section is devoted to showing several examples where \textit{Problem \ref{problem1}} can be explicitly solved.

\begin{example}\label{examples1}
In the following examples, $K\subseteq\R^n$ is a convex body with $0\in K$ and $g\in\mathbb H_d(\R^n)$ is the optimal polynomial verifying \eqref{ineq:set} for the given convex body $K$ given in \cite{lasserre}.
\begin{enumerate}
\item \label{ex:inf_at_1}
Let $f(x)=e^{-\|x\|_K}$. Then
$K_\lambda(f)=(\log(1/\lambda))K$,
$
H_t(f)=
%\bigcup_{\lambda\in(0,1]} \left(\frac{\log(1/\lambda)}{\log(t/\lambda)}\right) K=
\interior(K)
$
 for $t>1$, and
$
H_1(f)=
%\bigcup_{\lambda\in(0,1)} \left(\frac{\log(1/\lambda)}{\log(1/\lambda)}\right)K=
K.
$
Therefore $g_t=g$ for every $t\geq 1$. Thus
$\phi(t)=t|G_1(g)|$ and then $\min_{t\geq 1}\phi(t)=\phi(1).$

\item
More generally, for $\alpha>1$ let $f(x)=e^{-\|x\|_K^\alpha}$. Then $K_\lambda(f)=(\log(1/\lambda))^{1/\alpha}K$, $H_t(f)=
%\bigcup_{\lambda\in(0,1]}\frac{(\log 1/\lambda)^{1/\alpha}}{(\log t/\lambda)}K
\alpha^{-1/\alpha}(\alpha'\log t)^{-1/\alpha'}K
$  for $t>1$ ($\alpha^{-1}+(\alpha')^{-1}=1$)
and $H_1(f)=
%\bigcup_{\lambda\in(0,1)} \frac{1}{(\log 1/\lambda)^{1/\alpha'}}K$
\R^n$.
 Then $g_t=\alpha^{d/\alpha}(\alpha'\log t)^{d/\alpha'} g$ for every $t> 1$ ($g_1$ is undefined since $H_1(f)$ is unbounded). Thus
$\phi(t)=t\alpha^{-n/\alpha}(\alpha'\log t)^{-n/\alpha'}|G_1(g)|$ and then $\min_{t> 1}\phi(t)=\phi(e^{n/\alpha'})=(e/n)^{n/\alpha'}\alpha^{-n/\alpha}|G_1(g)|.$

\item \label{ex:infty_at_1}
Let $f=\chi_K$. Then $K_\lambda(f)=K$, $H_t(f)=
%\bigcup_{\lambda\in(0,1]}\log(t/\lambda)^{-1} K=
(\log t)^{-1} K$  for $t>1$, and $H_1(f)=\R^n$.
Therefore $g_t=(\log t)^{d}g$ for $t>1$ ($g_1$ is undefined since $H_1(f)$ is unbounded).
Thus
$\phi(t)=t(\log t)^{-n}|G_1(g)|$ and
$\min_{t> 1}\phi(t)=\phi(e^n)=  (e/n)^{n} |G_1(g)|$.

\item Let
$$
f(x)=
\begin{cases} 1 & x\in K,
\\
e^{1-\Vert x\Vert_K} & \text{ otherwise.}
\end{cases}
$$
Then $f\in\mathcal F(\mathbb R^n)$. Moreover,
$$
K_\lambda(f)=
\begin{cases}
\left(1+\log\frac1\lambda\right)K & 0<\lambda<1
\\
K &  \lambda=1,\end{cases}
$$
and thus
$$
H_t(f)=
\begin{cases}
\frac{1}{\log t}K & 1\leq t\leq e
\\
\interior(K) & t>e. \end{cases}
$$
Hence,
$$
g_t=
\begin{cases}
(\log t)^d g &  1\leq t\leq e
\\
g &  t>e\end{cases}
$$
%defines the homogeneous polynomial of degree d minimizing $|G_1(g_t)|$ such that $H_t(f)\subseteq G_1(g_t)$,
and thus
$$
\phi(t)=\begin{cases}
\frac{t}{(\log t)^n}|G_1(g)| & 1\leq t\leq e
\\
t|G_1(g)| & t>e\end{cases}
$$
is not differentiable at the point $t=e$, precisely where it attains the  minimum $\min_{t\geq 1}\phi(t)=\phi(e)=e|G_1(g)|$.
\end{enumerate}\end{example}

\section{A new approach to approximate log-concave functions by polynomials}\label{section:prob2}

With the purpose of getting uniqueness for the optimal polynomial, we pose \textit{Problem \ref{problem2}}
as a similar minimization problem, where the polynomial exponent $\tfrac1d$ is dropped
in \eqref{ineq:func_poly1}, turning into \eqref{ineq:func_poly2}.
%pairs $(r,g)$
%with compact support.
%such that $r\geq 0$ and $g\in\F_d(\R^n)$ with $G_1(g)$ bounded and $f(x)\leq e^{r-g(x)}$ for every $x\in\R^n$.

In order to solve \textit{Problem \ref{problem2}}, for $t\ge1$,
we introduce
$$
\widehat H_t(f)
=\bigcup_{\lambda\in(0,1)}(\log(t/\lambda))^{-1/d}K_\lambda(f).
$$

This case is not a generalization of Lasserre's problem, and moreover we can not assure $\widehat H_t(f)$ to be bounded, as in \textit{Problem \ref{problem1}} (as one can see by taking  $f(x)=e^{-\|x\|_2}$). In fact, the following example shows the existence of a log-concave function $f$ for which $\widehat H_1(f)$ is an unbounded set with finite volume.

\begin{example}\label{example:unboundedfinite}
Let us consider $f(x)=(1-\|(x_1,\dots,x_{n-1},0)\|_\infty)\chi_{[0,1]^n}(x)$. Notice that $f$ is integrable and  concave in its support, and thus, also log-concave. Moreover, $K_\lambda(f)=[0,1-\lambda]^{n-1}\times[0,1]$ for every $\lambda\in(0,1)$. Thus
\[
\widehat{H}_1(f)=\bigcup_{0<\lambda<1}
\left[0,(1-\lambda)(\log1/\lambda)^{-1/d}\right]^{n-1}\times\left[0,(\log1/\lambda)^{-1/d}\right].
\]
Notice that the terms of the union when $\lambda\rightarrow 1^-$ contain points with arbitrarily large norm, thus  $\widehat{H}_1(f)$ is unbounded.

The function $h(\lambda)=(1-\lambda)(\log(1/\lambda))^{-1/d}$ fulfills
$
h'(\lambda)=-\frac{d\lambda\log(1/\lambda)+\lambda-1}{d\lambda\left(\log1/{\lambda}\right)^{1+1/d}}.
$
%Thus $h'(\lambda)=0$ if and only if $\lambda_0=-1/d\omega(-\frac{e^{-1/d}}{d})$, where $$\omega:[-1/e,+\infty)\to[-1,+\infty)$$ is the Lambert $W$ function, i.e., the inverse of
%the strictly increasing function $\omega\in[-1,+\infty)\mapsto\omega e^\omega$. Moreover, $h$ is increasing in $(0,\lambda_0)$ and decreasing in $(\lambda_0,1)$. Therefore

Let  $\lambda_d\in(0,1)$ be the unique root of the equation $h'(\lambda)=0$ in $(0,1)$.
Then, $h$ is increasing in $(0,\lambda_d)$ and decreasing in $(\lambda_d,1)$. Therefore
\[
\begin{split}
& \widehat{H}_1(f)=
\left[0,(1-\lambda_d)(\log 1/\lambda_d)^{-1/d}\right]^{n-1}
\times
\left[0,(\log1/\lambda_d)^{-1/d}\right]\cup
\\
&\bigcup_{\lambda_d\le\lambda\le1}
\left\{
(x_1,\dots,x_{n-1},(\log 1/\lambda)^{-1/d})
:
0\leq x_i\leq(1-\lambda)(\log 1/\lambda)^{-1/d},\,
1\le i\le n-1
\right\}.
\end{split}
\]
Since the first term in the union above is bounded, 
 $|\widehat{H}_1(f)|<+\infty$  if and only if the second term in the union has finite volume. Letting $\mu=(\log(1/\lambda))^{-1/d}$,  that term  becomes
\[
\{(x_1,\dots,x_{n-1},\mu):0\leq x_i\leq \mu(1-e^{-\mu^{-d}}),\,\mu\geq \mu_d\}
\]
where $\mu_d=(\log(1/\lambda_d))^{-1/d}$. Using Fubini's formula,  its volume is
\[
\int_{\mu_d}^{+\infty}\mu^{n-1}\left(1-e^{-\mu^{-d}}\right)^{n-1}d\mu
= \int_{0}^{1/\mu_d}\frac{\left(1-e^{-\theta^d}\right)^{n-1}}{\theta^{n+1}}d\theta,
\]
%Since the last integral is finite in any interval of the form $[\theta_0,+\infty)$, for $\theta_0>0$, then we only study the behavior of the function inside the integral in a neighbourhood of $0$, which, by its Taylor expansion, becomes
%\[
%%\int_0^{\theta_0}\frac{\left(1-e^{-\theta^d}\right)^{n-1}}{\theta^{n+1}}d\theta=
%\int_0^{\theta_0}\frac{(\theta^d+o(\theta^{2d}))^{n-1}}{\theta^{n+1}}d\theta=
%\int_0^{\theta_0}\theta^{(n-1)(d-1)-2}+o(\theta^{(n-1)(d-1)-1})d\theta.
%\]
This last integral converges if and only if $(n-1)(d-1)>1$, which %due to $n\geq 2$, $d\geq 2$, 
turns out to be always true except in the case $n=d=2$. In this last case, $\widehat{H}_1(f)$ is unbounded with infinite volume. 
Otherwise $|\widehat{H}_1(f)|<+\infty$, as desired.

Note also that $\widehat H_1(f)$ is not convex, in contrast to $H_1(f)$ (see Lemma \ref{lem:polar}),
while $f$ is concave on its compact support. 
Regarding \textit{Problem \ref{problem2}}, it would be interesting to solve Lasserre's problem for these type of sets $\widehat H_1(f)$ (bounded or not).
\end{example}

For that reason, we will restrict the study of \textit{Problem \ref{problem2}} to $\mathcal B(\R^n)$, the set of all log-concave functions for which
$\widehat H_1(f)$ is bounded. Note that $\mathcal B(\R^n)\subset \mathcal F(\R^n)$. Since $K_\lambda(f)$ are convex sets that contain the origin, $\widehat H_t(f)$ is decreasing in $t\ge1$, so the boundedness of  $\widehat H_t(f)$ for any $t\ge1$ is guaranteed by the condition $f\in\mathcal B(\R^n)$.

%Let $\mathcal C(\R^n)$ be the set of integrable log-concave functions
%with compact support with either $\widehat{H}_1(f)$ bounded or $|\widehat{H}_1(f)|=+\infty$ whenever $\widehat{H}_1(f)$ is unbounded.

Similar lemmas to those given in Section \ref{section:prob1} are now provided. The proofs follow the same ideas as in the previous section.

\begin{lemma}\label{lem:equiv2}
Let $f:\R^n\to[0,+\infty)$ be a log-concave function with $\Vert f\Vert_\infty=f(0)=1$,
$d\in\mathbb N$ even,
$g\in\mathbb H_d(\mathbb R^n)$ and $t\ge1$.
The following are equivalent:
\begin{enumerate}
\item $f(x)\leq t\exp(-g(x))$ for all $x\in\R^n$.
\item $\widehat H_t(f)\subset G_1(g)$.
\end{enumerate}
\end{lemma}

\begin{lemma}\label{lem:log t d increasing}
Let $f:\R^n\to[0,+\infty)$ be a log-concave function with $\Vert f\Vert_\infty=f(0)=1$ and
$d\in\mathbb N$ even. Then for every $1<t_0<t_1$ we have that
$$
(\log t_0)^\frac1d \widehat H_{t_0}(f)
\subset
(\log t_1)^\frac1d \widehat H_{t_1}(f).
$$
\end{lemma}

%\begin{proof}
%Since $f$ is log-concave, then $K_\lambda(f)$ is convex for every $\lambda\in(0,1]$. Moreover, since $f(0)=\Vert f\Vert_\infty=1$,
%then $0\in K_\lambda(f)$ for every $\lambda\in(0,1]$ and thus $K_\lambda(f)$ are star-shaped with respect to $0$. Hence
%\begin{align*}
%H'_{t_0}(f)
%&=
%\bigcup_{\lambda\in(0,1]}\frac{1}{\left(\log t_0+\log(1/\lambda)\right)^\frac1d}K_\lambda(f)
%\\
%&=
%\left(\frac{\log t_1}{\log t_0}\right)^\frac1d\bigcup_{\lambda\in(0,1]}\frac{1}{\left(\log t_1+\frac{\log t_1}{\log t_0}\log(1/\lambda)\right)^\frac1d}K_\lambda(f)
%\\
%&\subset
%\left(\frac{\log t_1}{\log t_0}\right)^\frac1d\bigcup_{\lambda\in(0,1]}\frac{1}{(\log t_1+\log(1/\lambda))^\frac1d}K_\lambda(f)
%\\
%&=
%\left(\frac{\log t_1}{\log t_0}\right)^\frac1d H'_{t_1}(f).
%\end{align*}
%\end{proof}

\begin{rmk} In these two lemmas we have not assumed $f\in\mathcal{B}(\R^n)$ or the integrability of $f$. Moreover, in Lemma \ref{lem:equiv2}  the log-concavity of $f$ is only used in the case $t=1$; indeed, in both lemmas the only fact needed is that $\alpha_1 K_\lambda(f)\subset \alpha_2 K_\lambda(f)$ for any $0<\alpha_1\le\alpha_2$, which is equivalent to the fact that $K_\lambda(f)$ is star-shaped with respect to the origin.
Finally, the inclusions in Lemma \ref{lem:log t d increasing} above are sharp (take $f=\chi_K$ for a convex body $K\subset\R^n$).
\end{rmk}

%However, $\widehat H_1(f)$ can be unbounded as in \textit{Problem \ref{problem1}}. Like in \textit{Problem \ref{problem1}}, if $\widehat H_t(f)$ is unbounded, its volume is infinity.
%\end{lemma}

%%(((Give example for which $\widehat{H}_t(f)$ is not convex?)))

For every $t\ge 1$,
the minimum in \textit{Problem \ref{problem2}} over $g\in\H_d(\R^n)$ such that $\widehat H_t(f)\subset G_1(g)$  is attained.
Indeed, we may apply the minimization problem solved by Lasserre to $\widehat H_t(f)$ (actually its closure) and get $\widehat g_t\in\H_d(\R^n)$  the only polynomial  verifying
$$
\widehat H_t(f)\subset G_1(\widehat g_t)
$$
with minimum volume $|G_1(\widehat g_t)|$ among all $g\in\H_d(\R^n)$ such that $\widehat H_t(f)\subset G_1(g)$.

Then, the infimum in \textit{Problem \ref{problem2}} can be rewritten as
\[
\inf_{t\geq 1}t|G_1(\widehat g_t)|.
\]
For any $t\ge1$, let $\widehat v(t)=|G_1(\widehat g_t)|$  and $\widehat \phi(t)=t\widehat v(t)$ be the function to be minimized.

\begin{lemma}\label{lem:propofv2}
Let $f\in\mathcal B(\R^n)$ with $\Vert f\Vert_\infty=f(0)=1$ and  $d\in\mathbb N$ even.
Then $\widehat v$ is a decreasing function and $(\log t)^\frac{n}{d} \widehat v(t)$ is increasing on $t$.
As a consequence,  $\widehat \varphi$ is increasing on $[ e^{\frac{n}{d}},+\infty)$.
\end{lemma}

%\begin{proof}
%By Lemmas \ref{lem:equiv2} and \ref{lem:log t d increasing} we have that
%\[
%H'_{t_0}(f)\subset\left(\frac{\log t_1}{\log t_0}\right)^\frac1d H'_{t_1}(f)\subset
%\left(\frac{\log t_1}{\log t_0}\right)^\frac1d G_1(g_{t_1}) = G_1\left(\frac{\log t_0}{\log t_1}g_{t_1}\right).
%\]
%Since $(\log t_0/\log t_1)g_{t_1}\in \mathbb H_d(\R^n)$ then the minimality of $g_{t_0}$ implies that
%\[
%|G_1(g_{t_0})|\leq|G_1\left(\frac{\log t_0}{\log t_1}g_{t_1}\right)|=\left(\frac{\log t_1}{\log t_0}\right)^\frac{n}{d}|G_1(g_{t_1})|,
%\]
%as desired.
%
%
%The statement is a consequence of the fact that
%\[
%\varphi(t)=\frac{t}{(\log t)^\frac{n}{d}}(\log t)^\frac nd |G_1(g_t)|
%\]
%is a product of the increasing functions $(\log t)^{\frac{n}{d}}|G_1(g_t)|$ (if $t\geq 1$)
%and $\frac{t}{(\log t)^\frac nd}$ (if $t\geq e^{\frac{n}{d}}$).
%\end{proof}

\begin{lemma}\label{lem:log-convex-varphi}
Let $f\in\mathcal B(\R^n)$ with $\Vert f\Vert_\infty=f(0)=1$ and let  $d\in\N$ even. Then
$r\in [0,+\infty)\mapsto \widehat \phi(e^r)$ is log-convex. Moreover, if
$\widehat \phi(t_0)=\widehat \phi(t_1)=\min_{t\geq 1}\widehat{\phi}(t)$, then $t_0=t_1$.
\end{lemma}

\begin{proof}
We have that for any $t_0,t_1\in[1,+\infty)$, $f(x)\leq \exp(r_i-\widehat g_{t_i}(x))$ for $i=0,1$, where $t_i=e^{r_i}$. Then
\[
\begin{split}
f(x)&\leq\left(\exp(r_0-\widehat g_{t_0}(x))\right)^{1-\theta}
\left(\exp(r_1-\widehat g_{t_1}(x))\right)^{\theta}
\\
&=\exp((1-\theta)r_0+\theta r_1-((1-\theta)\widehat g_{t_0}(x)+\theta \widehat g_{t_1}(x))).
\end{split}
\]
Since $(1-\theta)\widehat g_{t_0}+\theta\widehat  g_{t_1}\in\mathbb H_d(\R^n)$,
the minimality of $\widehat g_{\exp((1-\theta)r_0+\theta r_1)}$ implies that
$|G_1(\widehat g_{\exp((1-\theta)r_0+\theta r_1)})|\leq|G_1((1-\theta)\widehat g_{t_0}+\theta \widehat g_{t_1})|$.
Therefore, using Hölder's inequality and Lemma \ref{lem:IntegralOfg},
\[
\begin{aligned}
 |G_1(\widehat g_{\exp((1-\theta)r_0+ \theta r_1)})|
 &\le
 |G_1((1-\theta) \widehat g_{t_0}+ \theta  \widehat g_{t_1})|
 \\
 &=\Gamma(\tfrac{n}{d}+1)^{-1}
 \int_{\R^n} e^{-((1-\theta) \widehat g_{t_0}(x)+ \theta \widehat  g_{t_1}(x))}\,dx
 \\
 &\le
 \left(\Gamma(\tfrac{n}{d}+1)^{-1}\int_{\R^n} e^{-\widehat g_{t_0}(x)}\,dx\right)^{1-\theta}
 \left(\Gamma(\tfrac{n}{d}+1)^{-1}\int_{\R^n} e^{-\widehat g_{t_1}(x)}\,dx\right)^{\theta}
 \\
 &=
 |G_1(\widehat g_{t_0})|^{1-\theta}|G_1(\widehat g_{t_1})|^{\theta},
\end{aligned}
\]
and thus $\widehat{\phi}(\exp((1-\theta)r_0+\theta r_1))\leq \widehat{\phi}(e^{r_0})^{1-\theta}\widehat{\phi}(e^{r_1})^{\theta}$,
i.e., $\widehat\phi(e^r)$ is log-convex for $r\in[0,+\infty)$.

Let us now assume that $0\leq r_0\leq r_1$, $t_i=e^{r_i}$ are such that $\widehat\phi(e^{r_0})=\widehat\phi(e^{r_1})=\min_{r\geq 0}\widehat{\phi}(e^r)$.
This implies that for every $\theta\in[0,1]$,
$$
|G_1(\widehat g_{\exp((1-\theta)r_0+\theta r_1)})|=|G_1(\widehat g_{t_0})|^{1-\theta}|G_1(\widehat g_{t_1})|^\theta,
$$
 which by H\"older's equality
cases means that $e^{-\widehat g_{t_0}}=ce^{-\widehat g_{t_1}}$, for some $c>0$. Since $\widehat g_{t_i}(0)=0$ for $i=0,1$,
then $c=1$, thus $\widehat g_{t_0}=\widehat g_{t_1}$, and hence $|G_1(\widehat g_{t_0})|=|G_1(\widehat g_{t_1})|$ from which
we get that $t_0=e^{r_0}=e^{r_1}=t_1$, concluding the proof.
\end{proof}

%\begin{thm}
%\begin{equation}\label{}
%\begin{split}
%\text{Given }f\in\mathcal C(\mathbb R^n) & \text{ with }||f||_\infty=f(0)=1, \text{ and }d\in\mathbb N\text{ even, let}\\
%\rho=\inf \int_{\mathbb R^n} te^{-g(x)}dx\quad & \text{s.t.} \\
%                  & f(x)\leq te^{-g(x)}dx\\
%                  & g\in\mathbb H_d(\mathbb R^n)\\
%                  & t\geq 1
%\end{split}
%\end{equation}
%Then there exists unique $g_0\in\mathbb H_d(\R^n)$ and $t_0\geq 1$ such that
%\[
%\int_{\R^n}t_0e^{-g_0(x)}dx=\rho.
%\]
%\end{thm}
%
%

Now we prove the existence and uniqueness of a global minimum for \textit{Problem \ref{problem2}}.

\begin{proof}[Proof of Theorem \ref{thm:prob2}]
By Lemma \ref{lem:log-convex-varphi}, $\widehat\phi(e^r)=e^r|G_1(\widehat g_t)|$ is a log-convex function, and thus convex in $[0,+\infty)$. This shows, in particular, that $\widehat{\phi}$
is continuous in $(1,+\infty)$.

The same ideas used in the proof of Theorem \ref{thm:prob1} can be used to see that
\begin{equation*}
\lim_{t\to1^+}\widehat v(t)=\widehat v(1).
\end{equation*}

By Lemma \ref{lem:propofv2}, $\widehat\phi$ is increasing on $[e^\frac nd,+\infty)$,
and since $\widehat\phi$ is continuous, it attains its minimums in $t_0\in[1,e^{\frac nd}]$.

Finally, Lemma \ref{lem:log-convex-varphi} shows that if $\widehat\phi$ attains the minimum,
it must be at a single point $t=t_0$,  concluding the proof.
\end{proof}

Now we can characterize the minimization point using the Karush-Kuhn-Tucker conditions (see \cite{APE}, \cite{hiriart}).
In order to do so, we first show a global convexity property of the function to be minimized.

\begin{lemma}\label{lemma:global_convex} The feasible set $\F_d(\R^n)$ is convex and
the objective function  $W:[0,+\infty)\times\F_d(\R^n)$ given by
$$W(r,g)=e^r\int_{\mathbb R^n}\exp(-g(x))dx$$
is log-convex and strictly convex.
\end{lemma}

\begin{proof}
Let $(r_i,g_i)\in[0,+\infty)\times\F_d(\mathbb R^n)$, $i=0,1$.
Then,
Hölder's inequality % and Lemma \ref{lem:IntegralOfg}
 implies that
\begin{equation}\label{eq:holder_in_phi}
\int_{\R^n} e^{-((1-\theta) g_0(x)+ \theta  g_1(x))}\,dx
\le
 \left(\int_{\R^n} e^{-g_0(x)}\,dx\right)^{1-\theta}
 \left(\int_{\R^n} e^{-g_1(x)}\,dx\right)^{\theta},
\end{equation}
thus showing that  $\F_d(\R^n)$ is convex and
$$
W((1-\theta)(r_0,g_0)+\theta(r_1,g_1))\leq W(r_0,g_0)^{1-\theta}W(r_1,g_1)^\theta,
$$
 and hence the
log-convexity of $W$.

Notice that the Arithmetic-Geometric mean inequality implies that $W$ is convex.

Furthermore, we show now that it is strictly convex.
First of all, let us suppose  $(r_0,g_0),(r_1,g_1)$ and $\theta\in[0,1]$ with
$$
W((1-\theta)(r_0,g_0)+\theta(r_1,g_1))=(1-\theta)W(r_0,g_0)+\theta W(r_1,g_1).
$$
The equality case of the AG-mean inequality directly
implies that $W(r_0,g_0)=W(r_1,g_1)$. Moreover, it also means that
there is equality in \eqref{eq:holder_in_phi}. Hence, the equality case of Hölder's inequality
implies the existence of $c>0$ such that
$e^{-g_0}=c e^{-g_1}$.
Since $g_0(0)=g_1(0)=0$,  $c=1$. Therefore $g_0=g_1$. Thus
$|G_1(g_0)|=|G_1(g_1)|$, and then $r_0=r_1$,
hence showing the strict convexity of $W$.
\end{proof}

\begin{proof}[Proof of Theorem \ref{thm:touching2}]
Let $W$ be   defined  as in Lemma \ref{lemma:global_convex}.
\textit{Problem \ref{problem2}} then rewrites as the following minimization problem:
\[
\min_{(r,g)\in C} W(r,g),
\]
where $$C=\{(r,g)\in[0,+\infty)\times \F_d(\R^n):
r-\log f(x)-g(x)\geq 0  \text{ for all } x\in S_f\}$$
with $S_f=\{x\in\R^n:f(x)\not=0\}$.

Any $g\in\mathbb H_d(\mathbb R^n)$ can be uniquely written as $g=\sum_{\alpha\in\N_d^n}g_\alpha x^\alpha$, so we can identify each $g$
with its coordinate vector $(g_\alpha)_\alpha\in\R^{h_d(n)}$.
%, where $h_d(n)={n+d-1\choose d}$.
Notice that $r-g(x)=r-\sum_{\alpha\in\N_d^n} g_\alpha x^\alpha=\langle(r,(g_\alpha)_\alpha),(1,-(x^\alpha)_\alpha)\rangle$.
Thus the feasible set can be rewritten as
\begin{equation}\label{eq:Domain_C}
\begin{split}
C=\{&(r,(g_\alpha)_\alpha)\in[0,+\infty)\times \R^{h_d(n)}:
\\
&g\in\F_d(\R^n),\,
\langle(r,(g_\alpha)_\alpha),(1,-(x^\alpha)_\alpha)\rangle\geq\log f(x)
\text{ for all } x\in S_f
\},
\end{split}
\end{equation}
so it is convex, as it is the intersection of half-spaces.

Assume condition (i) holds. %By Theorem \ref{thm:prob2},
Notice that, taking $t_2=e^{r_2}$,
%the set
%$$R=\{(r,g)\in[0,+\infty)\times \F_d(\R^n):W(r,g)\leq W(r_2,g_2)\}$$
%has non empty interior and $C\cap R=\{(r_2,g_2)\}$.
%Hence
$(r_2,g_2)\in\partial C$. Otherwise, we can take $(r,g_2)\in C$ with $r<r_2$, and $W(r,g_2)<W(r_2,g_2)$ contradicting that $W$ attains its minimum on $C$ at $(r_2,g_2)$.

Since $C$ is described in \eqref{eq:Domain_C} as  intersection of halfspaces,
then the supporting cone $S_C(r_2,g_2)$ of $C$ at $(r_2,g_2)$ is
given by the set of all such halfspaces whose boundaries contain $(r_2,g_2)$, i.e.,
\[
\begin{split}
 S_C(r_2,g_2)=&\{(r,(g_\alpha)_\alpha)\in[0,+\infty)\times \R^{h_d(n)}:
 \\
& g\in \F_d(\R^n),\,
\langle (r,(g_\alpha)_\alpha),(1,-(x^\alpha)_\alpha)\rangle\geq \log f(x)
\text{ for all }x\in S^*_f\},
\end{split}
\]
where $S^*_f=\{x\in S_f: r_2-g_2(x)=\log f(x)\}$. Thus we have that
\[
N_C(r_2,g_2)=\mathrm{pos}(\{(-1,(x^\alpha)_\alpha):x\in S^*_f\}),
\]
where $N_C(z)$ is the outer normal cone of $C$ at $z$, for every convex set $C$ and every $z\in\partial C$,
and $\text{pos}(R)$ is the positive hull of $R$, the smallest convex cone containing $R$.

Since $W$ is a differentiable  strictly convex function, and $C$ is a convex set,
 under the assumption
$(r_2,g_2)\in(0,+\infty)\times\mathrm{int}(\F_d(\R^n))$ the Karush-Kuhn-Tucker conditions (see \cite{APE}) characterize $(r_2,g_2)$ by
\begin{equation}
\label{eq:touchpoints}
-\nabla W(r_2,g_2)\in N_C(r_2,g_2).
\end{equation}
Besides, by \eqref{eq:DerivativeOfPsi}
$$
\nabla W(r,g)= \left(W(r,g),\left(-e^r\int_{\mathbb R^n}x^\alpha\exp(-g(x))dx\right)_\alpha\right).
$$

Moreover, since $N_C(r_2,g_2)\subset\R^{h_d(n)+1}$ is a convex cone, using Carath\'eodory's theorem for cones,
the previous condition \eqref{eq:touchpoints}
is equivalent
to the existence of $x_1,\dots,x_m\in\mathbb R^n$, $m\leq h_d(n)+1$, with
$r_2-\log f(x_i)-g_2(x_i)=0$, and $\lambda_i>0$, $1\leq i\leq m$, such that
\[
\left(-W(r_2,g_2),\left(e^r\int_{\mathbb R^n}x^\alpha\exp(-g(x))dx\right)_\alpha\right)
=
\sum_{i=1}^m\lambda_i(-1,(x_i^\alpha)_\alpha)
\]
which proves (ii).

Conversely, suppose condition (ii) holds. The Karush-Kuhn-Tucker conditions imply that $(r_2,g_2)$ is an extreme point,
and thus by the convexity of $W$, a local minimization point of $W$ on $C$.
Since $W$ is strictly convex and $C$ is a convex set, this local minimization point must be the only global minimization point, and (i) is proved.
\end{proof}

\begin{rmk} Note that our arguments work for the case $ |\widehat H_1(f)|=+\infty$ (as long as $\widehat H_t(f)$ is bounded for some $t>1$), since then the minimum is attained at some $t_2\in(1,+\infty)$. This remark allows us to apply our results to a more general set of functions outside $\mathcal B(\R^n)$, as shown in the following example. The only case we can not use our arguments is when  $\widehat H_1(f)$ is unbounded but $ |\widehat H_1(f)|<+\infty$ (see Example \ref{example:unboundedfinite}). For this reason, it would be very interesting to get an extension of Lasserre's theorem for sets of the form $\widehat H_1(f)$.
\end{rmk}

\begin{example}
Let $f(x)=\exp(-\|x\|_2^d)$ for some $d\in\N$ even. \textit{Problem \ref{problem2}} then makes sense
for $f$ for every even $d'\in\{2,\dots,d\}$. Since
\[
\begin{split}
\widehat{H}_t(f) & =\bigcup_{0<\lambda<1}\frac{1}{(\log t-\log\lambda)^\frac{1}{d'}}K_\lambda(f)\\
& =\bigcup_{0<\lambda<1}\frac{(-\log\lambda)^\frac1d}{(\log t-\log\lambda)^\frac{1}{d'}}B^n_2.
\end{split}
\]
Note that $\widehat H_1(f)=\R^n$.
Since the maximum of $(-\log\lambda)^\frac1d/(\log t-\log\lambda)^{\frac{1}{d'}}$ is attained at
$\lambda_M=t^{-\frac{d'}{d-d'}}$, then
\[
\widehat{H}_t(f)=
\frac{d'^{\frac1d}}{d^{\frac{1}{d'}}}
\left(\frac{d-d'}{\log t}\right)^{\frac{1}{d'}-\frac1d}B^n_2=G_1(g_t),
\]
where
\[
g_t(x)=\frac{d}{d'^{\frac{d'}{d}}(d-d')^{1-\frac{d'}{d}}}(\log t)^{1-\frac{d'}{d}}\|x\|_2^{d'}.
\]
Then
\[
\begin{split}
\min_{t\geq 1}t|G_1(g_t)| & =
\frac{d'^{\frac{n}{d}}}{d^{\frac{n}{d'}}}(d-d')^{n(\frac{1}{d'}-\frac1d)}\omega_n\min_{t\geq 1} \frac{t}{(\log t)^{n(\frac{1}{d'}-\frac1d)}},
\end{split}
\]
where $\omega_n=|B_2^n|$. The minimum above is attained at $t_2=e^{n(\frac{1}{d'}-\frac{1}{d})}$, and
then
\[
f(x)\leq t_2 e^{-g_{t_0}(x)}
=e^{n(\frac{1}{d'}-\frac{1}{d})}
%\exp\left(-\frac{d}{d'^{\frac{d'}{d}}(d-d')^{1-\frac{d'}{d}}}(\log t_0)^{1-\frac{d'}{d}}\|x\|_2^{d'}\right)
\exp\left(-\frac{d^\frac{d'}{d}}{d'}{n}^{1-\frac{d'}{d}}\|x\|_2^{d'}\right)
\]
is the unique solution to \textit{Problem \ref{problem2}} for $f$ with
$$
t_2\int_{\mathbb R^n}\exp(-g_{t_0}(x))dx
%=\Gamma(\tfrac{n}{d'}+1)\frac{d'^{\frac{n}{d}}}{d^{\frac{n}{d'}}}(d-d')^{n(\frac{1}{d'}-\frac1d)}\omega_n\frac{t_0}{(\log t_0)^{n(\frac{1}{d'}-\frac1d)}}.
=
\omega_n\Gamma(\tfrac{n}{d'}+1)
\left.
\left(\frac{n}{ed}\right)^{n/d}
\middle/
\left(\frac{n}{ed'}\right)^{n/d'}
\right. .
$$
\end{example}

\section{Application: $d$-outer volume and integral ratio}\label{sec:Application_OVR_d_oir_d}

Given a compact set $K\subset\R^n$, it is a natural question to consider how well does the volume of the level set of the $d$-Lasserre L\"owner polynomial approximates the volume of $K$. In the context of convex bodies,  $K\in\mathcal K^n$ 
(resp. centrally symmetric convex bodies $K\in\mathcal K^n_0$), it was already Ball in \cite[Thms. 1$\&$2]{Ball1991} who showed, by means of the Brascamp-Lieb inequality, that
the largest ratio between the volumes of a compact convex set $K$ and its John ellipsoid is attained
when $K$ is a simplex (resp. a cube when $K\in\mathcal K^n_0$). 
Later on, Barthe (see \cite[Thms. 2$\&$3]{Barthe1998_2}) 
showed, by means of a reverse Brascamp-Lieb inequality \cite{Barthe1998}, that, in the case of the L\"owner ellipsoid, the analogous largest ratio
between the volume of $G_1(g_2)$ and the volume of $K$ (assuming $G_1(g_2)$ is the L\"owner elliposid of $K$) is attained when $K$ is a centered simplex (resp. a crosspolytope when $K\in\mathcal K^n_0$) .

The existence of the $d$-Lasserre-L\"owner polynomial $g_d$ naturally leads to define the \emph{$d$-outer volume  ratio}
$\mathrm{o.v.r}_d(K)$ for any given $K\in\mathcal K^n$ as 
\[
\mathrm{o.v.r}_d(K)=\left(\frac{|G_1(g_d)|}{|K|}\right)^{1/n},
\]
for every even $d\in\mathbb N$.

Since $g_d$ is homogeneous of degree $d$, $g_d$ is an even function, and thus $G_1(g_d)$ is a centrally symmetric star-shaped with respect to the origin set. 
The first non-trivial examples on how well we can approximate $K\in\mathcal K^n_0$ by $G_1(g_d)$ were computed by Lasserre (see \cite[Thm. 3.4]{lasserre}), for the $2$-dimensional cube in the cases $d=4$ and $d=6$.

 %Benko and Kro\'o \cite{BenkoKroo2009} can be applied here for our purposes. 
Benko and Kro\'o showed (see Theorem 2 and Lemma 5 in \cite{BenkoKroo2009}) that if 
$K\in\mathcal K^n_0$ has $C^{1+\varepsilon}$ boundary, for some $\varepsilon\in(0,1]$,
then for any $\tau\in(0,1)$ and any even degree $d$ there exists a sequence of polynomials $g_d\in\mathbb H_d(\mathbb R^n)$
such that $|g_d(x)-1|\leq c d^{-\tau\varepsilon}$, for every $x\in\partial K$ and some constant $c>0$, only depending on $K$. By the homogeneity of $g_d$ and of the Minkowski gauge $\|\cdot\|_{K}$, the inequality above can be rewritten as
$$
\forall x\in\R^n\quad
(1-c d^{-\tau\varepsilon})\|x\|_K^d
\le
g_d(x)
\le
(1+c d^{-\tau\varepsilon})\|x\|_K^d.
$$
This inequality leads to the following theorem:

\begin{thm}\label{thm:approx_large_d}
Let $K\in\mathcal K^n_0$. Then $\displaystyle\lim_{d\to+\infty}\mathrm{o.v.r}_d(K)=1$.
\end{thm}

\begin{proof} 
Fix $\delta>1$. A standard approximation argument gives us some $Q\in\mathcal K^n_0$ of $C^2$ boundary,
with $K\subset Q$ and $(|Q|/|K|)^{1/n}\leq \sqrt{\delta}$. Let us apply Benko and Kro\'o result above to $Q$ (with $\varepsilon=1$ and any fixed $\tau\in(0,1)$) to get, for any even $d\geq 2$ a sequence of homogeneous polynomials $g_{d,\delta}\in\mathbb H_d(\mathbb R^n)$ and a constant $c_\delta>0$
such that 
$$
(1-c_\delta d^{-\tau})\|x\|_Q^d.
\le
g_d(x)
\le
(1+c_\delta d^{-\tau})\|x\|_Q^d\quad\forall x\in\R^n.
$$

Define $\overline{g}_{d,\delta}=(1+c_\delta d^{-\tau})^{-1} g_{d,\delta}\in \H_d(\R^n)$. We have $\overline{g}_{d,\delta}(x)\le1 $ for every $x\in Q$, which means  $Q\subset G_1(\overline{g}_{d,\delta})$.

On the other hand, if $x\in G_1(g_{d,\delta})$, then
$\|x\|_Q\le(1-c_\delta d^{-\tau})^{-1/d}$, which implies $G_1(g_{d,\delta})\subset (1-c_\delta d^{-\tau})^{-1/d}Q$.

Using the fact that $|G_1(\overline{g}_{d,\delta})|=
(1+c_\delta d^{-\tau})^{{n/d}}|G_1(g_{d,\delta})|$,

\[
\begin{split}
\left(\frac{|G_1(\overline{g}_{d,\delta})|}{|K|}\right)^\frac1n 
&= 
(1+c_\delta d^{-\tau})^{\frac{1}{d}} 
\left(
\frac{|Q|}{|K|}
\frac{|G_1(g_{d,\delta})|}{|Q|}
\right)^\frac1n
\\
& \leq 
\sqrt{\delta}
\left(\frac{1+c_\delta d^{-\tau}}{1-c_\delta d^{-\tau}}\right)^\frac1d \leq \delta,
\end{split}
\]
for any $d\ge d_\delta$ and large enough even $d_\delta$.
Since $K\subset Q \subset G_1(\overline{g}_{d,\delta})$, choosing the sequence $\delta_k=\frac{k+1}{k}$ and taking  $d_{\delta_{k+1}}>d_{\delta_{k}}$, the sequence of polynomials
$$
g_{d_2,2},g_{d_2+2,2},\dots,g_{d_{3/2},3/2},g_{d_{3/2}+2,3/2},\dots
$$
immediately proves the result.
\end{proof}

\begin{rmk}
 Rogers and Shephard showed  (see \cite{RS}) that if $K\in\mathcal K^n$ with $0\in K$, then $|\mathrm{conv}(K\cup(-K))|\leq 2^n|K|$. Considering $g^{(1)}_d$ and $g_d^{(2)}$
the $d$-Lasserre-L\"owner polynomials of $K$ and $\mathrm{conv}(K\cup(-K))$, respectively, we have that $K\subset \mathrm{conv}(K\cup(-K))\subset G_1(g_d^2)$, and 
\[
\textrm{o.v.r}_d(K)=\left(\frac{|G_1(g^{(1)}_d)|}{|K|}\right)^{\frac1n} \leq 2\left(\frac{|G_1(g_d^{(2)})|}{|\mathrm{conv}(K\cup(-K))|}\right)^\frac1n.
\]
Therefore, if $K\in\mathcal K^n$ with $0\in K$, we have that $$\limsup_{d\to+\infty}\textrm{o.v.r}_d(K)\le2.$$
\end{rmk}

A natural functional extension  of the $d$-outer volume ratio for any  $f\in\mathcal F(\mathbb R^n)$ with $\Vert f\Vert_\infty=f(0)=1$ is the \textit{$d$-outer integral ratio}
$$
\textrm{o.i.r}_d(f)=\left(\left.
t\int_{\R^n} \exp(-g(x)^{1/d})\,dx\right/
\int_{\R^n}f(x)\,dx
\right)^{1/n}
$$
where $(t,g)$ minimizes 
\textit{Problem \ref{problem1}}.
For $d=2$, a similar definition is considered in \cite{IT}.

Theorem \ref{thm:approx_large_d} can also be extended to some examples whenever we approximate log-concave functions.
For instance, we can show that
if $f\in\mathcal{F}(\R^n)$, then
\[
\lim_{d\to+\infty}
\mathrm{o.i.r}_d(f) = 1
\]
whenever 
$f(x)=e^{-\|x\|_K}$ with $K\in\mathcal K^n_0$. Indeed, it was shown in  Example \ref{examples1}$(1)$ that $H_t(f)=\mathrm{int} (K)$ if $t>1$ and $H_1(f)=K$.
Since $K\in\mathcal K^n_0$, we can take $g_d$ a sequence of homogeneous polynomials given by Theorem \ref{thm:approx_large_d} such that
$K\subset G_1(g_d)$ with
$|G_1(g_d)|/|K|\rightarrow 1$ when $d\rightarrow+\infty$.
Since $\min_{t\geq 1}t|G_1(g_d)|=|G_1(g_d)|$, and $H_1(f)=K\subset G_1(g_d)$ for every even $d\geq2$,
we have that $f(x)\leq \exp(-g_d(x)^{1/d})$ (see Lemma \ref{lem:equiv}) and
\[
\mathrm{o.i.r}_d(f)
\leq
\left(\left.
\int_{\mathbb R^n} \exp(-g_d(x)^{1/d})\,dx
\right/
\int_{\mathbb R^n}f(x)\,dx
\right)^\frac1n
=\left(\frac{|G_1(g_d)|}{|K|}\right)^\frac1n \rightarrow 1
\]
as $d\rightarrow+\infty$ (see also Lemma \ref{lem:IntegralOfg}).

\end{document}